\documentclass[reqno]{amsart}%
\usepackage[a4paper, margin=3cm]{geometry}
\usepackage{amssymb,enumerate,amsmath,bbm,bm,cite,amsfonts,mathrsfs,array,esint,cases,mathtools}
\usepackage{graphics,color,eso-pic,graphicx,bbding,rotating,verbatim}
\usepackage{tikz,makecell}
\usepackage[colorlinks=true,linkcolor=blue,citecolor=red,urlcolor=red, backref=page]{hyperref} 
\usepackage{cite}
\usepackage{esint} 
\usepackage{stmaryrd}
\usepackage{indentfirst}
\usepackage{pifont}
\makeatletter

\newcommand{\Rmnum}[1]{\expandafter\@slowromancap\romannumeral #1@}
\makeatother

\theoremstyle{plain}
\newtheorem{hyp}{Hypothesis}
\newtheorem{definition}{Definition}[section]

\newtheorem{claim}{Claim}

\theoremstyle{definition}

\newtheorem{theorem}{Theorem}[section]
\newtheorem{corollary}[theorem]{Corollary}
\newtheorem{lemma}[theorem]{Lemma}

\theoremstyle{remark}
\newtheorem{remark}{Remark}

\numberwithin{equation}{section}

\def\rn{\mathbb{R}^n}

\allowdisplaybreaks

\author[J. Tan]{Jiawei Tan}
\address{Jiawei Tan:
	School of Mathematics and Computing Science \\
	Guilin University of Electronic Technology \\
	Center for Applied Mathematics of Guangxi (GUET) \\
	Guangxi Colleges and Universities Key Laboratory of Data Analysis and Computation \\
	Guilin 541004 \\
	People's Republic of China}
\email{jwtan@guet.edu.cn}

\author[J. Wang]{Jiahui Wang}
\address{Jiahui Wang:
	School of Mathematical Sciences \\
	Beijing Normal University \\
	Laboratory of Mathematics and Complex Systems \\
	Ministry of Education \\
	Beijing 100875 \\
	People's Republic of China}
\email{wangjh0121@163.com}

\author[Q. Xue]{Qingying Xue}
\address{Qingying Xue:
	School of Mathematical Sciences \\
	Beijing Normal University \\
	Laboratory of Mathematics and Complex Systems \\
	Ministry of Education \\
	Beijing 100875 \\
	People's Republic of China}
\email{qyxue@bnu.edu.cn}


\thanks{The first author is supported by the Science and Technology Project of Guangxi (Guike AD23023002), and the third author  was partly supported by the National Key R\&D Program of China (No. 2020YFA0712900) and NSFC (No. 12271041).
}

\date{\today}
\title[ ]
{\bf Boundedness of a class of multilinear operators and their iterated commutators on Morrey-Banach function spaces}

\begin{document}
\maketitle
\begin{abstract}
	This paper investigates the boundedness of a broad class of operators within the framework of generalized Morrey-Banach function spaces. This class includes multilinear operators such as multilinear 
$\omega$-Calderón-Zygmund operators, multilinear maximal singular integral operators, multilinear pseudo-differential operators, and multilinear square functions, as well as linear operators such as rough singular integral operators, nonintegral operators, and Stein's square functions. 
The boundedness of this class of operators and their commutators on Morrey-Banach spaces is established provided they satisfy either a pointwise sparse domination assumption or the $W_{r}$ property (Lerner, Lorist and Ombrosi, 
Math. Z. 2024), which significantly generalize the classical theory of Morrey spaces.	As an application, the boundedness of iterated commutators of this class of multilinear operators is extended to Morrey-Lorentz spaces and Morrey spaces with variable exponents. 
	
\end{abstract}

\section{Introduction}The primary objective of this paper is to develop a unified framework for establishing the boundedness of a broad class of multilinear operators and their associated commutators on Morrey-Banach function spaces. This framework encompasses, as special cases, Morrey-Lorentz spaces and Morrey spaces with variable exponents, thereby providing a comprehensive extension of classical results in the theory of Morrey-type spaces. We begin with a brief review of the research in the Morrey type function spaces.

\subsection{Morrey type spaces}{It will be  divided into three parts.}
\vspace{0.2cm}

\noindent\textbf{(1). Classical Morrey spaces.}
In the 1930s, Wiener's work on the limit of functions 
$f$  satisfying
$$\frac{1}{s^{1-\alpha}} \int_{0}^{s} |f(x)|^{p} dx, \quad \alpha\in (0,1), \; p=1 \; \text{or} \; p=2$$
led to the development of local Morrey spaces. It was later recognized that local Morrey spaces could be generalized to the high-dimensional case with the norm 
\begin{align}\label{m-1}
	\|f\|_{B^{p}} = \sup_{r>0}\frac{1}{|B(0,r)|^{\frac{1}{p}}} \|f\chi_{B(0,r)}\|_{L^{p}(\rn)}
\end{align}
Notice that the classical Morrey space is obtained by replacing 
$B(0,r)$ in \eqref{m-1}  with
$B(x,r)$ where $x\in \rn$
, in the following manner:

\begin{align}\label{morrey}
	L^{p,\lambda}(\mathbb{R}^n)\coloneqq\left\{f\in L_{\text{loc}}^p(\mathbb{R}^{n}): \|f\|_{L^{p,\lambda}(\mathbb{R}^n)} = \sup_{x\in \mathbb{R}^n, r>0} \frac{1}{r^{\lambda/p}} \|f\chi_{B(x,r)}\|_{L^p(\mathbb{R}^n)}<\infty\right\}
\end{align}
for $1\leq p<\infty$ and $0\leq \lambda \leq n$. This space was further developed by Morrey \cite{mor} in 1938 to address certain problems in quasi-elliptic partial differential equations. When $\lambda>n$, $L^{p,\lambda}(\rn) = \{0\}$ is trivial.
 Morrey spaces are expected to enjoy nicer properties than Lebesgue spaces, by the reason that  Morrey spaces generalize Lebesgue spaces by introducing a more refined control over the behavior of functions, especially in local settings.
Indeed, for $\lambda=0$ and $\lambda=n$, the Morrey spaces $L^{p,0}(\rn)$ and $L^{p,n}(\rn)$ coincide with the Lebesgue spaces $L^{p}(\rn)$ and $L^{\infty}(\rn)$, respectively.
Moreover, the Morrey spaces are complete in the sense of norm.
As is well known, the Morrey spaces have other equivalent formulations. For instance, we can express them as
$$ \|f\|_{M^{p,q}(\rn)} = \sup_{x\in \mathbb{R}^n, r>0} r^{\frac{n}{q}-\frac{n}{p}} \|f \chi_{B(x,r)}\|_{L^{p}(\rn)},$$
where $1\leq p\leq q\leq \infty$. In fact, by defining
$\lambda \equiv n\left(1-\frac{p}{q}\right),$
we obtain the equivalence 
 $\|f\|_{L^{p,\lambda}(\rn)} = \|f\|_{M_{p,q}(\rn)}$. In the rest of the paper, we mainly follow the notation in \eqref{morrey}.

Since 1938, many results established in Lebesgue spaces have been extended to Morrey spaces. In 1975, Adams \cite{ada} was the first to prove the boundedness of the Riesz potential on Morrey spaces. In 1987, Chiarenza and Frasca \cite{chi} showed that the Hardy-Littlewood maximal operator is bounded on any Morrey space $L^{p,\lambda}(\rn)$ with $1<p<\infty, 0<\lambda<n$.
Subsequently, it was demonstrated that the Calder\'on-Zygmund operators and $\omega$-Calder\'{o}n-Zygmund operators are both bounded linear operators on $L^{p,\lambda}(\rn)$ for $0\leq\lambda<n$. Moreover, it was shown by Olsen \cite{ols} that the following
H\"{o}lder's inequality holds in Morrey spaces.
$$\|fg\|_{L^{r,\nu}(\rn)} \leq \|f\|_{L^{p,\lambda}(\rn)} \|g\|_{L^{q,\mu}(\rn)},$$
where $f\in L^{p,\lambda}(\rn)$ and $g\in L^{q,\mu}(\rn)$ with $1\leq p,q<\infty, \frac{1}{p}+\frac{1}{q}\geq 1$ and
$\frac{1}{r} =\frac{1}{p}+\frac{1}{q},\quad \frac{\nu}{r} =\frac{\lambda}{p} +\frac{\mu}{q}.$
\vspace{0.2cm}

\noindent\textbf{(2).  Generalized Morrey spaces.}
In the past decade, significant efforts have been made to extend the study of Morrey spaces to more general classes of Morrey spaces. Specifically, $\|\cdot\|_{L^{p}(\rn)}$ in \eqref{morrey} is replaced by other more general norms. For instance, in 2012, Ragusa \cite{rag} defined the Morrey-Lorentz spaces, where $\|\cdot\|_{L^{p,q}(\rn)}$ takes the place of $\|\cdot\|_{L^{p}(\rn)}$, and showed some embedding properties. 
Nakai \cite{nak2} introduced such spaces, and gave basic properties of them and a necerassary and sufficient condition for the boundedness of the Hardy-Littlewood maximal operator from Morrey-Orlicz space to another. 
Unlike the definition given by Nakai, Sawano et al. \cite{sawa5} defined another Morrey-Orlicz norm 
and presented some inequalities for generalized fractional integral operators on such space are established. 
Consider Morrey spaces with variable exponents, Almeida, Hasanov and Samko \cite{alm} introduced the spaces $L^{p(\cdot),\lambda(\cdot)}(\rn)$ equipped with the norm 
$$\|f\|_{L^{p(\cdot),\lambda(\cdot)}(\rn)} = \inf\left\{\nu>0: \rho_{p(\cdot),\lambda(\cdot)}\left(\frac{f}{\nu}\right)\leq 1\right\}<\infty,$$
with the modular 
$\rho_{{p(\cdot),\lambda(\cdot)}}(f) \coloneqq \sup_{x\in \rn,~ r>0} \frac{1}{r^{\lambda(x)}} \|f\chi_{B(x,r)}\|_{L^{p(\cdot)}(\rn)},$
where $L^{p(\cdot)}(\rn)$ is given in Section \ref{pr+4}.
In the case of a bounded $\Omega$, several equivalent norms and the embedding theorems for such Morrey spaces under the assumption that $p(x)$ satisfies the $\log$-condition were also presented in \cite{alm}. 
When the underlying space is a homogeneous-type space $(X,\rho,\mu)$, Kokilashvili and Meskhi \cite{kok
} introduced the space $M_{p(\cdot)}^{q(\cdot)}$ with the norm 
$$\|f\|_{M_{p(\cdot)}^{q(\cdot)}} = \sup_{x\in X,~ r>0} \left(\mu(B(x,r))\right)^{\frac{1}{p(x)}-\frac{1}{q(x)}} \|f\chi_{B(x,r)}\|_{L^{q(\cdot)}(X)},$$
where $1<\inf_{X} q\leq q(\cdot) \leq p(\cdot) \leq \sup_{X} p <\infty$. Similar, in the case of a bounded $X$, several equivalent norms and the embedding theorems were also obtained. 
Read \cite{ho1} for more content about Morrey spaces with variable exponents. 
For more information about these spaces, we refer the readers to \cite{ho1, ayk, sawa5} and references therein. 
\vspace{0.2cm}

\noindent\textbf{(3). Morrey-Banach spaces.}
In 2017, Ho \cite{ho4} extended the extrapolation theory to Morrey-Banach spaces, which include Morrey-Lorentz spaces, Morrey-Orlicz spaces and Morrey spaces with variable exponents, thus deriving the John-Nirenberg inequalities on Morrey-Banach spaces and the characterizations of BMO in terms of Morrey-Banach spaces. However, the extrapolation theory in \cite{ho4} only serves the boundedness of the nonlinear commutator $Nf \coloneqq T(f \log |f|) - Tf \log|Tf|$ on some subspaces of the Morrey-Banach spaces, since the density of bounded functions with compact supports on Morrey-Banach spaces remains unknown, where $T: \mathscr{S}(\rn) \to \mathscr{S}'(\rn)$ is a Calder\'on-Zygmund operator
$$Tf(x)= \int_{\rn} K(x,y) f(y) dy,$$
with the kernel function $K$ satisfying the classical size condition and some smoothness conditions.
Later on, Ho \cite{ho3} refined the extrapolation theory and established the boundedness of the nonlinear commutator $Nf$ on Morrey-Banach spaces. 
In order to adapt to Morrey-Banach spaces, Ho \cite{ho2} generalized the definitions of classical Calder\'on-Zygmund operator and gave the boundedness of singular integrals and their commutators, and the weak type estimates of singular integral operators. 
More work about the Morrey-Banach space can be found in \cite{zha} and references therein.
\subsection{Multilinear operators}
It is well known that the multilinear operator theory can be traced back to the celebrated works of Coifman and Meyer \cite{coi4} on singular integral operators such as commutators and pseudo-differential operators. This theory plays a crucial role in harmonic analysis, providing not only essential techniques, such as compensated compactness, for the study of partial differential equations but also powerful tools for addressing long-standing open problems. For example, the prominent Kato square root problem is one of the challenges that can benefit from the application of multilinear operator theory.
The study of multilinear operator theory, usually involves decomposing multiple variables simultaneously, which is naturally much more complex than linear analyses. This complexity also endows it with more profound significance and value.

Since the beginning of this century, the multilinear operator theory has achieved remarkable progress. For instance, multilinear Calder\'{o}n-Zygmund operators were systematically studied by Grafakos and Li \cite{gra1}, Grafakos and Torres \cite{ gra3}. Tao, Vargas, and Vega \cite{tao} employed bilinear methods to investigate multipliers and applied these techniques to the resolution of the restriction conjecture and the Kakeya conjecture-two of the four major open problems in Harmonic analysis.  Moreover, through the Fourier restriction method, Bourgain \cite{bou} transformed the well-posedness problem of nonlinear dispersive equations into the task of establishing bilinear or multilinear estimates in specific spacetime Banach spaces. This innovation has provided a new approach to studying the well-posedness of nonlinear evolution equations, particularly in the context of low-regularity problems.
The theory of multiple $A_{\vec{p}}$ -weights was  established by Lerner, P\'erez, et al. \cite{ler}. For more related works, we refer the readers to \cite{ler1,cao,bp+tan,che,li2} and the references therein.


Observe that significant progress has also been made in the study of operators which beyond the scope of multilinear Calderón-Zygmund theory, such as multilinear pseudo-differential operators \cite{cao}, Calder\'{o}n commutators \cite{DL2019} and Stein's square functions \cite{cho}. 
These operators enjoy several properties analogous to those of  the Calder\'{o}n-Zygmund operators,  such as the same $L^p$ boundedness, endpoint estimate, and sparse domination. In particular, Sparse domination provides a powerful framework for obtaining quantitative weighted estimates of operators.
In 2008, Beznosova \cite{bez} proved the $A_{2}$ conjecture for the dyadic paraproduct, which, in combination with Hyt\"onen’s dyadic representation theorem, leads to Hyt\"onen’s proof of the full $A_{2}$ conjecture in \cite{hyt2}. 
Until 2016, it was illustrated that a Calder\'on-Zygmund operator $T$ can be dominated pointwise by a finite number of sparse operators $\mathcal{T}_{\mathcal{S}}$ in \cite{con}, that is,
$$|Tf(x)|\leq c_{n,T} \sum_{j=1}^{3^{n}} \mathcal{T}_{\mathcal{S}_{j}}f(x),$$
where
$\mathcal{T}_{\mathcal{S}}f(x) = \sum_{Q\in \mathcal{S}} \langle |f| \rangle_{Q} \chi_{Q}(x),$
$\langle f\rangle_{Q}\coloneqq \frac{1}{|Q|} \int_{Q} f(x)dx$ and $\mathcal{S}$ is a sparse family of cubes from $\rn$, which is introduced in Section \ref{spare}.
By \cite{ler1}, the sparse operator and its adjoint adapted to the commutator with locally integrable function $b$ are defined by
\begin{align*}
	\mathcal{T}_{\mathcal{S},b}f(x) \coloneqq \sum_{Q\in \mathcal{S}} |b(x)-\langle b\rangle_{Q}| \langle|f|\rangle_{Q} \chi_{Q}(x),\quad
	\mathcal{T}_{\mathcal{S},b}^{\ast}f(x) \coloneqq \sum_{Q\in \mathcal{S}} \langle |b(x)-\langle b\rangle_{Q}| |f|\rangle_{Q} \chi_{Q}(x).
\end{align*}
\subsection{Motivation}Our motivation lies in two main aspects:
\begin{enumerate}
\item [{\bf $\bullet$}] It is well known that many operators, both convolutional and non-convolutional, enjoy certain kind of sparse domination. This includes non-convolutional operators such as multilinear 
$\omega$-Calderón-Zygmund operators, multilinear maximal singular integral operators, multilinear pseudo-differential operators, and multilinear square functions, among others. Given certain assumptions of sparse domination, it is natural to explore how the common features of these multilinear operators can be extracted and analyzed in order to establish their boundedness on various product function spaces.

In \cite{tan}, the first and third authors provided a partial answer  to this problem by studying the boundedness of multilinear operators and their commutators in rearrangement-invariant Banach spaces. These spaces form a large class of function spaces, including Lorentz and Orlicz spaces as notable examples. However, it is important to note that there is another class of spaces, namely Morrey spaces, which cannot be contained within rearrangement-invariant Banach spaces. Moreover, the classical Morrey space
$L^{p,\lambda}(\rn)$ for $0<\lambda<n$ are not even Banach spaces (\!\!\cite[Page 665]{sawa6}), nor are they quasi-Banach spaces. As a result, the study of the properties of multilinear operators and their iterated commutators in Morrey-type spaces has become a topic of great interest. Our first motivation is to  establish a unified method for the boundedness of a class of multilinear operators and their iterated commutators on Morrey-type spaces.
\item [{\bf $\bullet$}] 
In general, the pointwise sparse domination  usually requires the weak type $(1,1)$ for the grand maximal operator, which may be as large as the maximal truncation operators of the corresponding operators. 
However, whether the maximal truncation operators associated with many operators, such as rough singular integral operators with a kernel belongs to $L^\infty$, are of weak type $(1,1)$ remains an open question up to now \cite{ler10}.
Nevertheless, the remarkable works of Conde-Alonso et al. \cite{con1} and Lerner et al. \cite{ler9} indicate that pointwise sparse domination can be generalized to bilinear form sparse domination. The second motivation of this paper is to investigate whether the iterated commutators of rough singular integral operators exhibit boundedness on Morrey-Banach spaces and whether the bilinear form of sparse domination can yield analogous results.

\end{enumerate}
\subsection{Basic assumptions}

This paper aims to explore the boundedness of a class of multilinear operators on Morrey-Banach spaces, which is a very broad class of function spaces involving the classical Morrey spaces and Morrey-Lorentz spaces. 
We also consider the boundedness of rough singular integral operators on Morrey-Banach spaces. 

Before presenting our main results, we first introduce the necessary definitions and state the basic assumptions. Let $\mathcal{T}$ be a $m$-linear (or sublinear) operator from $\mathscr{X}_1 \times \cdots \times$ $\mathscr{X}_m$ into $\mathscr{Y}$, where $\mathscr{X}_1, \ldots, \mathscr{X}_m$ are some normed spaces and $\mathscr{Y}$ is a quasi-normed space. In our following statements, $\mathscr{X}_1, \ldots, \mathscr{X}_m$ and $\mathscr{Y}$ will be appropriately Lebesgue spaces.
\begin{definition}[\textbf {Iterated commutators of multilinear operators}]\label{def6.2}
	Given $m\in \mathbb{N}^{+}, l\leq m$,~$\vec{f} \coloneqq  \left(f_1, \ldots, f_m\right) \in \mathscr{X}_1 \times \cdots \times \mathscr{X}_m$, $\vec{b} \coloneqq  \left(b_{i_1}, \ldots, b_{i_l}\right)$ of measurable functions with $\{i_1,\ldots,i_l\}\subseteq \{1,\ldots,m\}$. The $m$-linear iterated commutators of $\mathcal{T}$ is given by
	$$
	\mathcal{T}_{\vec{b}}(\vec{f})(x) \coloneqq  \left[b_{i_1}, \cdots \left[b_{i_l}, \mathcal{T}\right]_{i_l} \cdots \right]_{i_1} (f_1, \dots, f_m)(x),
	$$
	where $[b_{i_l}, \mathcal{T}]_{i_l} (f_1, \dots, f_m) = b_{i_l} \mathcal{T}(f_1,\dots,f_m) - \mathcal{T}(f_1,\dots, b_{i_l} f_{i_l}, \dots, f_m)$. In particular, if $\mathcal{T}$ is an $m$-linear operator with a kernel representation of the form
	$$
	\mathcal{T}(\vec{f})(x) = \int_{\mathbb{R}^{nm}} K(x,\vec{y}) f_1(y_1) \cdots f_m(y_m) d\vec{y},
	$$
	where $d\vec{y} = dy_1 \cdots dy_m$, then $\mathcal{T}_{\vec{b}}$ can be represented as
	\begin{align}\label{ie+commutator}
		\mathcal{T}_{\vec{b}}(\vec{f})(x) = \int_{\mathbb{R}^{nm}} \prod_{s=1}^{l} (b_{i_s}(x) - b_{i_s}(y_{i_s})) K(x,\vec{y}) \prod_{s=1}^m f_s(y_s) d\vec{y}.
	\end{align}
\end{definition}
It is worth noting that the weighted strong and weak type endpoint estimates for the iterated commutators of the multilinear Calder\'{o}n-Zygmund operators in Definition \ref{def6.2} were established by P\'{e}rez et al. \cite{per2}. One can refer to \cite{cao,ler1} for more information about commutators.

We will consider two hypotheses that this class of operators need to satisfy. To introduce our first hypothesis, 
let $\mathcal{T}(b,f,Q,\gamma) $ be defined by
$$\mathcal{T}(b,f,Q,\gamma) \coloneqq  \begin{cases}
	|b-{\langle b\rangle}_{Q}|{\langle |f |\rangle}_{Q}, & \gamma=1, \\
	{\langle |(b-{\langle b \rangle}_{Q}) f|\rangle}_{Q}, &\gamma=2.
\end{cases}$$ 
The following hypothesis is crucial to our forthcoming discussion of $\mathcal{T}_{\vec{b}}$. 
\begin{hyp}[\textbf {Pointwise sparse domination}]\label{hyp}
	Let $I \coloneqq \{i_1,\ldots,i_l\}\subseteq \{1,\ldots,m\}$ and $\vec{b}=(b_{i_1}, \ldots,$ $b_{i_l})$ be locally integrable functions defined on $\mathbb{R}^n$. For an $m$-linear operator $\mathcal{T}$,  $\mathcal{T}_{\vec{b}}$ is its iterated commutator given by \eqref{ie+commutator}. Suppose that for any bounded functions $\vec{f}=\left(f_1, \ldots, f_m\right)$ with compact support, there exist $3^n$ sparse collections $\{\mathcal{S}_j\}_{j=1}^{3^n}$ such that for almost everywhere $x\in \mathbb{R}^n$, 
\begin{equation}\label{hyp2}
	|\mathcal{T}_{\vec{b}}(\vec{f})(x)| \leq C\sum_{j=1}^{3^n} \sum_{\vec{\gamma} \in \{1,2\}^l}\sum_{Q \in \mathcal{S}_j} \left(\prod_{s=1}^l \mathcal{T}(b_{i_s}, f_{i_s}, Q, \gamma_{i_s})\right) \left(\prod_{s \notin I} \left\langle\left|f_s\right|\right\rangle_Q\right) \chi_Q(x).
	\end{equation}
\end{hyp}
\begin{remark}
	When $I=\emptyset$,  inequality (\ref{hyp2}) in the above Hypothesis can be rewritten as
	$$
	|\mathcal{T}(\vec{f})(x)| \leq C\sum_{j=1}^{3^n} \sum_{Q \in \mathcal{S}_j}\prod_{s =1 }^m\left\langle\left|f_s\right|\right\rangle_Q\chi_Q(x), \qquad \hbox {\ for a.e. } x\in \mathbb{R}^n.
	$$
	It was known that the estimate in the form of Hypothesis \ref{hyp} holds for many operators, such as  multilinear Calder\'{o}n-Zygmund operators \cite[Theorem 1.4]{dam}, multilinear pseudo-differential operators \cite[Proposition 4.1]{cao}, etc. Therefore, all the results obtained in this paper are valid not only for the commutators of this class of operators but also for the operators themselves. However, for simplicity, the subsequent theorems will only be stated and demonstrated for the commutators.
\end{remark}

We proceed to formulate the second hypothesis, which applies to more general rough operators.
For $1\leq s\leq \infty$, let $T$ be a sublinear operator, and define the corresponding sharp grand maximal truncation operator by
$$M_{T,s}^{\sharp}f(x) \coloneqq \sup_{Q \owns x} \text{osc}_{s} \left(T(f\chi_{\rn\backslash{3Q}}); Q\right), \quad x\in \rn,$$
where  the supremum is taken over all $Q$ containing $x$ and
$$\text{osc}_{s}(f; Q)\coloneqq \left(\frac{1}{|Q|^2} \int_{Q\times Q} |f(x')-f(x'')|^{s} dx' dx''\right)^{\frac{1}{s}}.$$

Let $\varphi_{T,r}(\lambda)\coloneqq \lambda^{-\frac{1}{r}} \|T\|_{L^{r}(\rn)\to L^{r,\infty}(\rn)}$ for $ \lambda\in (0,1).$
The following hypothesis, taken from \cite{ler8}, describes the 
the $W_{r}$ property of $T$. 
\begin{hyp}[\textbf {$W_{r}$ property}]\label{hyp-new}
	Let $r\in [1,\infty)$.  An operator $T$ is said to satisfy the $W_{r}$ property, if there exists a nonincreasing function $\varphi_{T,r}: (0,1)\to [0,\infty)$ such that for any cube $Q$ and $f\chi_{Q} \in L^{r}(\rn)$, 
	$$\left|\left\{x\in Q: |T(f\chi_{Q})(x)| > \varphi_{T,r}(\lambda) \langle |f|\rangle_{r,Q}\right\}\right|
	\leq \lambda |Q|, \quad \lambda\in (0,1).$$
\end{hyp}
\vspace{0.01cm}
\subsection{Main results}
We first recall some necessary notations and concepts. Throughout this paper, for a locally integrable function $b$, $\|\vec{b}\|_{\text{BMO}}$ denotes the $m$-fold product of $\|b_i\|_{\text{BMO}}, i=1,\dots,m$, where $\|b_{i}\|_{\text{BMO}}\coloneqq\sup_{Q} \frac{1}{|Q|} \int_{Q}\left|b_{i}(x)-\langle b_{i} \rangle_Q\right|dx$.
Given a measurable set $E$, we denote $w(E)=\int_E w(x)dx$ and simply denote $|E|$ when $w=1$. 
Let $B(x,r)=\{z\in \mathbb{R}^n: |x-z|<r\}$ denote the open ball with center $x \in \mathbb{R}^n$ and radius $r>0$. We write $\mathcal{M}(\mathbb{R}^n,wdx)$ for the space of Lebesgue measurable functions on $(\mathbb{R}^n, wdx)$. The notation $\mathcal{M}(\rn)$ is reserved for $\mathcal{M}(\rn,dx)$.
We recall the definition of Morrey-Banach spaces from \cite[Definition 2.4]{ho2} as follows.
\begin{definition}[Morrey-Banach space]\label{def-mb}
	Fixed any Banach function space $X$ (see Section \ref{m-b} for the definition), the Morrey-Banach function space is defined by
	$$
	M_{X}^{u} = \Big\{ f \in \mathcal{M}(\mathbb{R}^n): \|f\|_{M_{X}^{u}} \coloneqq \sup_{x\in \mathbb{R}^n, r>0} \frac{1}{u(x,r)} \left\|f \chi_{B(x,r)}\right\|_{X} < \infty\Big\},
	$$
	with $u(x,r): \mathbb{R}^n \times (0,+\infty) \to (0,+\infty)$ be a Lebesgue measurable function.
\end{definition}

For any $0\leq \alpha < \infty$, we introduce the $\mathbb{W}_X^\alpha$ class in \cite{ho3} with some minor modifications.

\begin{definition}[$\mathbb{W}_X^\alpha$ class]\label{def-wxalpha}
	Let $X$ be a Banach function space. We say that a Lebesgue measurable function $u(x,r): \mathbb{R}^n \times (0,\infty) \to (0,\infty)$ belongs to $\mathbb{W}_X^\alpha$ if there exists a constant $C>0$ such that for any $x,x_1,x_2\in \mathbb{R}^n, r,r_1,r_2>0$ and $0\leq \alpha<\infty$, $u$ satisfies
	\begin{enumerate}
		\item $\frac{\|\chi_{B(x_1,r_1)}\|_{X}}{u(x_1,r_1)} \leq C \frac{\|\chi_{B(x_2,r_2)}\|_{X}}{u(x_2,r_2)}$, \quad if $u(x_1,r_1) \leq u(x_2,r_2)$,
		\item $\sum_{j=0}^{\infty} 2^{(j+1)\alpha} \frac{\|\chi_{B(x,r)}\|_{X}}{\|\chi_{B(x,2^{j+1}r)}\|_{X}} u(x,2^{j+1}r) \leq C u(x,r)$.
	\end{enumerate}		
\end{definition}

\begin{remark}
	Firstly, notice that $\mathbb{W}_{X}^{\alpha_2} \subseteq \mathbb{W}_{X}^{\alpha_1}$ for any $0\leq \alpha_1\leq \alpha_2<\infty$. When $\alpha=1$, $\mathbb{W}_{X}^1$ was studied in \cite{zha}. Secondly, item $(1)$ implies that $M_{X}^{u}$ is non-trival, and item $(2)$ ensures that the Hardy-Littlewood maximal operator $M$ is bounded on $M_{X}^{u}$.
\end{remark}

\subsubsection{\textbf{Results under the Hypothesis \ref{hyp}}}

We now present our first two main theorems. The first is the multilinear Coifman-Fefferman inequality in Morrey-Banach spaces.
\begin{theorem}\label{thm1}
	Let $I \coloneqq \{i_1,\dots,i_l\} = \{1,\dots,l\} \subseteq \{1,\dots,m\}$, $X$ be a Banach function space and $u: \mathbb{R}^n \times (0,+\infty) \to (0,+\infty)$ be Lebesgue measurable. Suppose that $\vec{b} \in \mathrm{BMO}^l$ and $\mathcal{T}_{\vec{b}}$ meets Hypothesis \ref{hyp}. If $0<p<\infty$, $X' \in \mathbb{M}$ and $u\in \mathbb{W}_{X}^0$, then there exists a constant $C>0$ such that for any $\vec{f}\in \mathcal{M}^m$ and $(\mathscr{M}_{L(\log L)} (\vec{f}))^p \in M_{X}^{u}$, $|\mathcal{T}_{\vec{b}}(\vec{f})|^p \in M_{X}^{u}$ and
	$$
	\left\| (\mathcal{T}_{\vec{b}} (\vec{f}))^p \right\|_{M_{X}^{u}} \leq  C\|\vec{b}\|_{\text{BMO}}^p \left\|M\right\|_{\mathfrak{B}_{X',u}}^{pl+\max\{2,p\}} \left\|\left(\mathscr{M}_{L(\log L)} (\vec{f}) \right)^p \right\|_{M_{X}^u},
	$$
	where $X' \in \mathbb{M}$ represents that the Hardy-Littlewood maximal operator $M$ is bounded on $X'$, as defined in Section \ref{m-b}.
\end{theorem}
\begin{remark}
	As applications of Theorem \ref{thm1}, the multilinear Coifman-Fefferman inequalities associated with the multilinear maximal singular integral operators and the multilinear Bochner-Riesz square functions in Morrey-Banach spaces are established in Theorems \ref{thm-maximal} and \ref{thm-square}, respectively. 
\end{remark}
\begin{remark}
It should be noted that Theorem \ref{thm1} is the unweighted version of \cite[Remark 1.7]{tan}, 
 but the Banach function space $X$ in Theorem \ref{thm1} is more general than it in \cite[Remark 1.7]{tan}, as $M_{X}^{u}=X$ for $u\equiv 1$. Furthermore, the Coifman-Fefferman inequalities for multilinear iterated commutators in general function spaces are new.
\end{remark}

The second main theorem is the boundedness of the iterated commutators of the multilinear operator $\mathcal{T}$ on Morrey-Banach spaces.
\begin{theorem}\label{thm2}
	Let $I \coloneqq \{i_1,\dots,i_l\} = \{1,\dots,l\} \subseteq \{1,\dots,m\}$, $X, X_1, \dots, X_m$ be Banach function spaces with $\|\chi_B\|_{X} \leq C\prod_{i=1}^m \|\chi_B\|_{X_i}$ for any ball $B$, and $X' \in \mathbb{M}$, $(X_i^p)'\in \mathbb{M}$. 
	Suppose that the multilinear Hardy-Littlewood maximal operator $\mathscr{M}(\vec{f})$ is bounded from $X_1^p \times \cdots \times X_m^p$ to $X^p$ with $1\leq p<\infty$, $\vec{b} \in \mathrm{BMO}^l$ and $\mathcal{T}_{\vec{b}}$ satisfies Hypothesis \ref{hyp}. Then for $u=\prod_{i=1}^m u_i$ with $u_i^{\frac{1}{p}} \in \mathbb{W}_{X_i^p}^0, i=1,\dots,m$, and $u^{\frac{1}{p}} \in \mathbb{W}_{X^p}^0$, the following statements are true.
	\begin{enumerate}
		\item when $p=1$ and $X_i \in \mathbb{M}, i=1,\dots,m$, 
		$$\left\| \mathcal{T}_{\vec{b}} (\vec{f}) \right\|_{M_{X}^{u}}
		\lesssim \|\vec{b}\|_{\text{BMO}} 
		\left\|M\right\|_{\mathfrak{B}_{X',u}}^{l+2} 
		\left\|\mathscr{M}\right\|_{M_{X_1}^{u_1} \times \cdots \times M_{X_m}^{u_m} \to M_{X}^{u}}
		\prod_{i=1}^m \|M\|_{M_{X_i}^{u_i} \to M_{X_i}^{u_i}} \|f_i\|_{M_{X_i}^{u_i}},$$
		where $\mathscr{M}$ is defined in Section \ref{weight}.
		
		\item when $p>1$ and $X_i' \in \mathbb{M}, i=1,\dots,m$,
		$$\left\| (\mathcal{T}_{\vec{b}} (\vec{f}))^p \right\|_{M_{X}^{u}}
		\lesssim \|\vec{b}\|_{\text{BMO}}^p
		\left\|M\right\|_{\mathfrak{B}_{X',u}}^{pl+\max\{2,p\}} 
		\left\|\mathscr{M}\right\|_{M_{X_1^p}^{u_1^{1/p}} \times \cdots \times M_{X_m^p}^{u_m^{1/p}} \to M_{X^p}^{u^{1/p}}}^p
		\prod_{i=1}^m \||f_i|^p\|_{M_{X_i}^{u_i}}.$$
	\end{enumerate}
\end{theorem}
\begin{remark}
	On the one hand, Theorem \ref{thm2} can yield a wealth of applications with respect to spaces, see Theorem \ref{thm3} for $u$-Morrey-Lorentz space, Corollary \ref{cor1} for Lorentz space, Theorem \ref{thm4} for Morrey space with variable exponent and Corollary \ref{cor2} for Lebesgue space with variable exponent. 
	On the other hand, as applications with respect to operators of Theorem \ref{thm2}, the boundedness for the iterated commutators of the multilinear maximal singular integral operators and the multilinear Bochner-Riesz square functions on Morrey-Banach spaces is established in Theorems \ref{thm-maximal} and \ref{thm-square}, respectively. 
\end{remark}
\begin{remark}
	We now make some comments on Theorem \ref{thm2}. Firstly, as mentioned earlier, the first and third authors \cite{tan} studied the boundedness of multilinear operators and their commutators in rearrangement invariant Banach spaces, which is generalized to the Morrey spaces in our Theorem \ref{thm2}. 
	Secondly, notice that multilinear Hardy-Littlewood maximal operator satisfies Hypothesis \ref{hyp}, which implies that Theorem \ref{thm2} is the essential generalization of \cite[Theorem 1.2]{zha}. Besides, compared to Theorem 1.2 in \cite{zha}, our conditions are weaker since  $\mathbb{W}_{X}^{1} \subseteq \mathbb{W}_{X}^{0}$. 
	Thirdly, when $m=1$, Theorem \ref{thm2} covers \cite[Theorem 4.3]{ho2}. To be more specific, Theorem 4.3 in \cite{ho2} requires the boundedness of the commutator $T_{b}$ on $X$, which is weakened in our Theorem \ref{thm2}.
	Lastly, for $0\leq\lambda<n$ and $1< q,q_1,\dots,q_m<\infty$ satisfying $\frac{1}{q} = \sum_{i=1}^{m} \frac{1}{q_i}$, we have  
	\begin{align*}
		\left\|\mathcal{T}_{\vec{b}}(\vec{f})\right\|_{L^{q,\lambda}(\mathbb{R}^n)}
		\leq C \|\vec{b}\|_{\text{BMO}} 
		\prod_{i=1}^m \|f_i\|_{L^{q_i,\lambda}(\mathbb{R}^n)}.
	\end{align*}
	where $L^{q,\lambda}(\mathbb{R}^n)$ is the classical Morrey space defined by \eqref{morrey}. In particular, if $\mathcal{T}$ is the multilinear Calder\'on-Zygmund operator, then \eqref{eq7} provides the boundedness for the multilinear Calder\'on-Zygmund operators and their iterated commutators on product Morrey spaces.
\end{remark}

\subsubsection{\textbf{Results under the Hypothesis \ref{hyp-new}}}
This part of our work focuses on the boundedness of operators satisfying the $W_{r}$ property on Morrey-Banach spaces. We aim to show that all operators in a certain class that meet the $W_{r}$ property, including rough singular integral operators, fractional integral operators with Hörmander kernels, intrinsic Littlewood-Paley square functions, non-integral operators outside the scope of Calder\'on–Zygmund theory and the associated square functions \cite{ler8}, are bounded on Morrey-Banach spaces. We first present the concept of $\mathbb{W}_{X,\delta}$, which is closely related to the boundedness for the Hardy-Littlewood maximal operator of exponential type. 
For any $0< \alpha\leq 1$, we write $u\in \mathbb{W}_{X,\delta}$ if $u$ satisfies
$$\sum_{j=0}^{\infty} \frac{\left\|\chi_{B(x,r)}\right\|_{(X^{\delta})'}^{\delta}}{\left\|\chi_{B(x,2^{j+1}r)}\right\|_{(X^{\delta})'}^{\delta}} u\left(x,2^{j+1}r\right) \leq Cu(x,r)$$
for some $C>0$. 
Next, we recall the iterated commutator of the sublinear operator. 
Given a sublinear operator $T$ and $b\in L_{\text{loc}}^{1}(\rn)$, the iterated commutator of $T$ is defined by 
$$T_{b}^{m}(f)(x) \coloneqq T\left((b(x)-b(\cdot))^{m} f\right)(x), \quad x\in \rn.$$
We now present our final main theorem under the $W_{r}$ property.
\begin{theorem}\label{thm-rough}
	Let $1\leq r<s\leq \infty$ and $m\in \mathbb{N}^{+}$. Suppose a sublinear operator $T$ and its sharp grand maximal truncation operator $M_{T,s}^{\sharp}$ enjoy Hypothesis \ref{hyp-new} (the $W_{r}$ property). Fixed $r_{0}> \max\{r,s'\}$, if $(X')^{\frac{1}{r_{0}}}, X^{\frac{1}{r_{0}}}$ are Banach spaces, $X^{\frac{1}{r_{0}}}, (X')^{\frac{1}{r_{0}}} \in \mathbb{M}$, and $u^{r_{0}}\in \mathbb{W}_{X^{{1}/{r_{0}}}}^{0}, u\in \mathbb{W}_{X',r_{0}^{-1}}$, then for any $b\in \text{BMO}$,
	$$\left\|T_{b}^{m} f\right\|_{M_{X}^{u}} \leq C \|b\|_{\text{BMO}}^{m} \|M\|_{M_{X^{1/r_{0}}}^{u^{r_{0}}}}^{\frac{1}{r_0}} \|M_{r_{0}}\|_{\mathfrak{B}_{X',u}} \|f\|_{M_{X}^{u}}.$$
\end{theorem} 
\begin{remark}
	We would like to make two remarks. On the one hand, as mentioned earlier, there are many operators that meet the $W_{r}$ property, such as rough singular integral operators, nonintegral operators falling outside the scope of Calder\'on–Zygmund theory and the associated square functions, see \cite[Remark 4.4]{ler9} for details. Therefore, Theorem \ref{thm-rough} holds for all these operators. On the other hand, note that Theorem \ref{thm-rough} can be obtained by the extrapolation in \cite{ho4}. Here, we provide an alternative proof, which also features a better constant.
\end{remark}
As an application, when $\Omega \in L^{\infty}(\mathbb{S}^{n-1})$ with zero average over the sphere,  the following rough homogeneous singular integral operators are taken into account, 
\begin{align}\label{rough}
	T_{\Omega} = \text{p.v.} \int_{\rn} f(x-y)\frac{\Omega(y/|y|)}{|y|^{n}} dy, \quad x\in \rn.
\end{align}Recall that, steamed from the research by Calder\'on-Zygmund \cite{cal1}, the rough homogeneous singular integral operator $T_{\Omega}$ has sparked extensive interest in the past few decades. In 1956, it was shown in \cite{cal2} that $T_{\Omega}$ is bounded on $L^{p}(\rn)$ for $1<p<\infty$ and $\Omega\in L\text{log}L(\mathbb{S}^{n-1})$ which was weakened to $\Omega\in H^{1}(\mathbb{S}^{n-1})$ in \cite{ric} where $H^{1}$ denotes the classical Hardy space. 
In \cite{see}, an illustrious outcome is that $T_{\Omega}$ is of weak type $(1,1)$ for $\Omega\in L\text{log}L(\mathbb{S}^{n-1})$. 
In the weighted case, Duoandikoetxea and Rubio de Francia \cite{duo1} derived that $T_{\Omega}$ is bounded on $L^{p}(w)$ for $w$ a Muckenhoupt $A_{p}$ weight, $1<p<\infty$ and $\Omega\in L^{\infty}(\mathbb{S}^{n-1})$. 
Later, the quantitative weighted estimation of $T_{\Omega}$ was established by dyadic decomposition in \cite{hyt1}. 
Recently, 
Lerner, Lorist and Ombrosi \cite{ler8} verified that the iterated commutator $(T_{\Omega})_{b}^{m}$ is bounded from $L^{p}(u)$ to $L^{p}(v)$ with $p>1$ and $u,v\in A_{p}$. 
To go a step further, the boundedness for the iterated commutator $(T_{\Omega})_{b}^{m}$ of $T_{\Omega}$ on Morrey-Banach spaces is stated in the corollary below.
\begin{corollary}
	Let $2 <s\leq \infty$ and $m\in \mathbb{N}^{+}$. Suppose $T_{\Omega}$ is a rough homogeneous singular integral operator defined by \eqref{rough}. Fixed $r_{0}> s'$, if $(X')^{\frac{1}{r_{0}}}, X^{\frac{1}{r_{0}}}$ are Banach spaces, $X^{\frac{1}{r_{0}}}, (X')^{\frac{1}{r_{0}}} \in \mathbb{M}$, and $u^{r_{0}}\in \mathbb{W}_{X^{{1}/{r_{0}}}}^{0}, u\in \mathbb{W}_{X',r_{0}^{-1}}$, then for any $b\in \text{BMO}$, 
	$$\left\|(T_{\Omega})_{b}^{m} f\right\|_{M_{X}^{u}} \leq C \|b\|_{\text{BMO}}^{m} \|M\|_{M_{X^{1/r_{0}}}^{u^{r_{0}}}}^{\frac{1}{r_0}} \|M_{r_{0}}\|_{\mathfrak{B}_{X',u}} \|f\|_{M_{X}^{u}}.$$
\end{corollary}

This paper is organized as follows: Section \ref{pre} is devoted to presenting the necessary definitions and some related lemmas for our theorems. In Section \ref{pr+1+2}, we prove the Coifman-Fefferman inequality (Theorem \ref{thm1}) and the boundedness of iterated commutators (Theorem \ref{thm2}) in Morrey-Banach spaces. Section \ref{pr+last} establishes the boundedness of rough singular integral operators on Morrey-Banach spaces (Theorem \ref{thm-rough}). Applications in $u$-Morrey-Lorentz spaces and Morrey spaces with variable exponents are discussed in Sections \ref{pr+3} and \ref{pr+4}, respectively. Finally, Section \ref{pr+5} provides applications of our theorems to certain other operators.

Throughout this paper, we always use $C$ to represent a positive constant, which is independent of the main parameters, but it may change at each occurrence. 
Let $\mathscr{S}(\rn)$ denote the set of all Schwartz functions on $\rn$, equipped with the classical well-known topology determined by a countable family of norms, and $\mathscr{S}'(\rn)$ its topological dual. 
If $a \leq C b$ ($a \geq C b,$ respectively) for any $a, b \in \mathbb{R}$, then we denote $a \lesssim b $ ($a \gtrsim b$, respectively) where $C$ is independent of $a$ and $b$. If $a \lesssim b \lesssim a$, then denote $a \simeq b$.

\section{Preliminary}\label{pre}
We begin by presenting some fundamental results on weights, Morrey-Banach spaces, Lorentz spaces, and sparse families.
\subsection{Weights}\label{weight}
\ \\
\indent Given $f\in L_{\textup{loc}}^1(\mathbb{R}^n)$, the Hardy-Littlewood maximal operator $M$ of $f$ is defined by
\begin{align}\label{h-l}
	Mf(x) \coloneqq \sup_{Q \owns x} \frac{1}{|Q|} \int_Q |f(y)|dy, \quad x\in \mathbb{R}^n,
\end{align}
where the supremum takes all the cubes that contain the point $x$. By a weight $w$ we mean a non-negative locally integrable function on $\rn$. The corresponding weighted Lebesgue spaces $L^{p}(w)$ are defined by
$$\|f\|_{L^{p}(w)} \coloneqq \int_{\mathbb{R}^{n}} \left|f(x)\right|^{p} w(x) dx.$$
In \cite{muc2}, Muckenhoupt pointed out that $M: L^p(w) \to L^p(w)$ for $1<p<\infty$ if and only if $w\in A_p(\rn)$, where $A_p(\rn)$ (simply denoted by $A_{p}$) refers to the classical Muckenhoupt weight, and we define the $A_p$ characteristic by
$$[w]_{A_p} \coloneqq \sup_{Q} \left(\frac{1}{|Q|} \int_{Q} w(y) dy \right) \left(\frac{1}{|Q|} \int_{Q} w(y)^{1-p'} dy \right)^{p-1} < \infty.$$
The constant $[w]_{A_p}$ is called the $A_p$ constant of $w$. Besides, $M: L^p(w) \to L^{p,\infty}(w)$ for $1\leq  p<\infty$ if and only if $w\in A_p$, where $w\in A_1$ means that 
\begin{align*}
	[w]_{A_{1}} &\coloneqq \sup_{Q} \left(\frac{1}{|Q|} \int_{Q} w(x) dx\right) \left\|\chi_{Q} w^{-1}\right\|_{L^{\infty}(w)}= \sup_{Q} \left(\frac{1}{|Q|} \int_{Q} w(x) dx \right) \left(\underset{x \in Q}{\operatorname{ess~inf}}~w(x)\right)^{-1} <\infty.
\end{align*}
Later on, Buckley \cite{buc} showed that 
$$[w]_{A_{p}}^{\frac{1}{p}} \leq \|M\|_{L^{p}(w)\to L^{p,\infty}(w)} \leq C_{n} [w]_{A_{p}}^{\frac{1}{p}},\quad \hbox{for \ } 1\leq p<\infty,$$
and
$$\|M\|_{L^{p}(w)\to L^{p}(w)} \leq C_{n} p' [w]_{A_{p}}^{\frac{1}{p-1}}, \quad \hbox{for \ } p>1.$$
Moreover, let $\sigma$ be the dual weight of $w$, that is, $\sigma=w^{1-p'}$ with $1<p<\infty$, then the $L^{p}(w)$ boundedness for the Hardy-Littlewood maximal operator is given by
$$\|M\|_{L^{p}(w)\to L^{p}(w)} \leq C_{n} \left(p'[w]_{A_{p}} [\sigma]_{A_{\infty}}\right)^{\frac{1}{p}}.$$

Next, we focus on other weights that will play crucial roles in our subsequent analysis.  According to the strictly increasing property of $A_{p}$ weights with respect to $p$, it is natural to define the $A_{\infty}$ weight,
$$A_{\infty} \coloneqq \bigcup_{p>1} A_{p}.$$
It is well known that a weight $w\in A_{\infty}(\rn)$ (simply denoted by $w\in A_{\infty}$) if and only if 
$$[w]_{A_{\infty}} \coloneqq \sup_{Q} \frac{1}{w(Q)} \int_{Q} M(w\chi_{Q}) (x) dx<\infty.$$

Now we consider the more complicated multiple weights. In \cite{ler}, Lerner et al. introduced the following multilinear Hardy-Littlewood maximal operator $\mathscr{M}$ defined by 
$$\mathscr{M}(\vec{f})(x) \coloneqq \sup_{Q\owns x}\prod_{i=1}^m \frac{1}{|Q|} \int_Q |f_i(y_i)|dy_i$$
for any ball $Q\subseteq\mathbb{R}^n$.
Let $m\in \mathbb{N}^{+}$. For $1 \leq  p_1, \dots, p_m <\infty$, $\frac{1}{p} = \frac{1}{p_1} + \cdots +\frac{1}{p_m}$ and weights $w_{1},\dots,w_{m}$, denote $\vec{p} = (p_1, \dots, p_m)$ and $\vec{w} = (w_1, \dots, w_m)$, and set $\nu_{\vec{w}} \coloneqq \prod_{i=1}^{m} w_i^{p/p_i}$, then
$$\mathscr{M}: L^{p_{1}(w_{1})} \times \cdots L^{p_{m}(w_{m})} \to L^{p,\infty}(\nu_{\vec{w}})$$
if and only if $\vec{w} \in A_{\vec{p}}$, where
\begin{equation*}
	[\vec{w}]_{A_{\vec{p}}} \coloneqq \sup_{Q} \left(\frac{1}{|Q|} \int_{Q} \nu_{\vec{w}}(x) dx \right) \prod_{i=1}^m \left(\frac{1}{|Q|} \int_{Q} w_i(x)^{1-p'_i} dx\right)^{\frac{p}{p'_i}} < \infty.
\end{equation*}
The quantity $[\vec{w}]_{A_{\vec{p}}}$ is called the $A_{\vec{p}}$ constant of $\vec{w}$. If there exists an index $i\in \{1,\dots,m\}$ such that $p_i = 1$, then $\left(\frac{1}{|Q|} \int_{Q} w_i(x)^{1-p'_i} dx \right)^{1/p'_i}$ is understood as $\left( \inf_{Q} w_i \right)^{-1}$. Besides, if $1<p_{1},\dots,p_{m}<\infty$, then the boundedness 
$$\mathscr{M}: L^{p_{1}}(w_{1}) \times \cdots \times L^{p_{m}}(w_{m}) \to L^{p}(\nu_{\vec{w}})$$
holds if and only if $\vec{w}\in A_{\vec{p}}$. 

A characterization of the multilinear $A_{\vec{p}}$ condition in terms of the linear $A_p$ classes was presented in \cite{ler}. 
If $\vec{w}=(w_1,\dots,w_m)$, $1\leq  p_1, \dots, p_m <\infty$ and $\vec{p} = \{p_1,\dots,p_m\}$, then $\vec{w} \in A_{\vec{p}}$ if and only if 
$$\left\{ 
\begin{aligned}
	&w_j^{1-p_j'} \in A_{mp_j'}, \quad j=1,\dots,m,\\
	&\nu_{\vec{w}} \in A_{mp},
\end{aligned}
\right.$$
where the condition $w_j^{1-p_j'} \in A_{mp_j'}$ in the case $p_j=1$ is understood as $w_j^{\frac{1}{m}} \in A_1$.
\subsection {BMO spaces}
Consider now the bounded mean oscillation function space $\text{BMO}(\rn)$ (simply denoted by BMO), which is a set of locally integrable functions on $\rn$ with 
$$\|f\|_{\text{BMO}} \coloneqq \sup_{Q} \frac{1}{|Q|} \int_{Q} |f(x)-\langle f\rangle_{Q} | dx <\infty,$$
where the supremum is taken over all cubes in $\rn$. The BMO space was first introduced by John and Nirenberg \cite{joh} to study  a class of nonlinear partial differential equation problems. However, it was not until Fefferman \cite{fef2} demonstrated in 1972 that BMO is the dual space of the Hardy space 
$H^{1}$ 
that BMO gained significant attention.
In many cases, BMO serves as a suitable alternative is to $L^{\infty}(\rn)$. However,  in reality, $L^{\infty}(\rn) \subsetneqq \text{BMO}$ and $\|b\|_{\text{BMO}} \leq 2\|b\|_{L^{\infty}(\rn)}$ for any $b\in L^{\infty}(\rn)$.

Given a locally integrable function $f$ on $\rn$, the sharp maximal operator $M^{\sharp}$ is defined by
$$M^{\sharp} f(x) \coloneqq \sup_{Q \owns x} \frac{1}{|Q|} \int_{Q} |f(y)-\langle f\rangle_{Q}| dy, \quad x\in \rn.$$
 The sharp maximal operator is closely related to BMO, a connection that is particularly evident in certain interpolation properties of various operators. To be more specific,
$$\text{BMO}(\rn) = \left\{f\in L_{\text{loc}}^{1}(\rn): M^{\sharp}f \in L^{\infty}(\rn)\right\}.$$
At this point, we have , we have $\|f\|_{\text{BMO}} = \left\|M^{\sharp}f\right\|_{L^{\infty}(\rn)}$.

The John-Nirenberg inequality from \cite{joh}, a fundamental result about the BMO function, is that for any $f\in \text{BMO}$, any cube $Q$ and $\alpha>0$, we have
$$\left|\left\{x\in Q: |f(x)-\langle f\rangle_{Q}| > \alpha\right\}\right| \leq e|Q| e^{-\frac{\alpha}{2^{n} e \|f\|_{\text{BMO}}}}.$$

The last topic about BMO is commutator. Given a linear operator $T$ defined on a set of measurable functions on $\rn$, if $b$ is a measurable function, then the corresponding commutator $[b,T]$ is a linear operator on measurable functions $f$ defined by
$$[b,T]f(x) \coloneqq b(x)Tf(x) - T(bf)(x), \quad x\in \rn.$$
The first result concerning the above commutator, as presented in \cite{coi3}, is that if $T$ is a classic singular integral operator with a smooth kernel and $b\in \text{BMO}$, then $[b,T]$ is bounded on $L^{p}(\rn)$ for $1<p<\infty$. Furthermore, for the Riesz transform $T$, the condition $b\in \text{BMO}$ is necessary. 
See \cite{coi3} for more information about commutators.

\subsection{Morrey-Banach function space}\label{m-b}
\ \\
\indent We first recall the definition of Banach function space \cite[Chapter 1, Definitions 1.1 and 1.3]{ben}. 
\begin{definition}[Banach function space, \cite{ben}]\label{def-bfs}
	A function space $X\subseteq \mathcal{M}(\mathbb{R}^n)$ is said to be a Banach function space on $\mathbb{R}^n$ if it satisfies
	\begin{enumerate}
		\item $\|f\|_X = 0 \Leftrightarrow f = 0$ a.e.,
		\item $|g|\leq |f|$ a.e. $\Rightarrow \|g\|_X\leq \|f\|_X$,
		\item $0\leq f_n \uparrow f$ a.e. $\Rightarrow \|f_n\|_X \uparrow \|f\|_X$,
		\item $\chi_E\in \mathcal{M}(\mathbb{R}^n)$ and $|E|<\infty \Rightarrow \chi_E\in X$,
		\item $\chi_E\in \mathcal{M}(\mathbb{R}^n)$ and $|E|<\infty \Rightarrow \int_{E} |f(x)|dx < C_E \|f\|_X$, for all $f\in X$ and some $C_E>0$.
	\end{enumerate}
\end{definition}

	The Lorentz spaces, the Orlicz spaces and the Lebesgue spaces with variable exponents are Banach function spaces (see \cite{ben,die} for details). According to item (5) of Definition \ref{def-bfs}, for any Banach function space $X$, we have $X\subseteq L_{\text{loc}}^{1}(\rn)$. For a Banach function space $X$ and $1\leq p<\infty$, we define  $X^p=\left\{f\in X: \|f\|_{X^p}\coloneqq \||f|^p\|_X^{{1}/{p}}<\infty\right\}$. 

The following associate space of $X$ originates from \cite[Chapter 1, Definitions 2.1 and 2.3]{ben}.
\begin{definition}[Associate space, \cite{ben}]\label{def-as-bfs}
	Suppose that $X$ is a Banach function space. The associate space of $X$, denoted by $X'$, is the collection of all Lebesgue measurable function $f$ such that
	$$\|f\|_{X'} = \sup\left\{\left|\int f(t)g(t)dt\right|: g\in X, \|g\|_X\leq 1\right\}<\infty.$$
\end{definition}
According to \cite[Chapter 1, Theorems 1.7 and 2.2]{ben}, if 
$X$ is a Banach function space, then
 $X'$ 
is also a Banach function space. The following theorem, extracted from \cite[Chapter 1, Theorem 2.4]{ben}, presents the H\"older  inequality for 
$X$ and $X'$.
\begin{lemma}(\cite{ben}).\label{holder-X}
	Let $X$ be a Banach space. Then for any $f\in X$ and $g\in X'$, we have
	$$\int_{\mathbb{R}^n} |f(x)g(x)|dx \leq \|f\|_X \|g\|_{X'}.$$
\end{lemma}
For any Banach function space $X$, we denote $X\in \mathbb{M}$ if the Hardy-Littlewood maximal operator $M$, as defined in  \eqref{h-l}, is bounded on $X$, and $X\in \mathbb{M}'$ if $M$ is bounded on $X'$.

Next, we recall the definition of Morrey-Banach spaces,
$$
M_{X}^{u} = \Big\{ f \in \mathcal{M}(\mathbb{R}^n): \|f\|_{M_{X}^{u}} \coloneqq \sup_{x\in \mathbb{R}^n, r>0} \frac{1}{u(x,r)} \left\|f \chi_{B(x,r)}\right\|_{X} < \infty\Big\},
$$
where $u(x,r): \mathbb{R}^n \times (0,+\infty) \to (0,+\infty)$ is a Lebesgue measurable function.
To present the extrapolation theory for Morrey-Banach spaces, Ho \cite{ho4} introduced the block space associated with the Banach function space.
\begin{definition}[block space,\cite{ho4}]
	Let $X$ be a Banach function space and $u: \mathbb{R}^n \times (0,\infty) \to (0,\infty)$ be a Lebesgue measurable function. Any Lebesgue measurable function $b$ is an $(X, u)$-block, written as $b\in  \mathfrak{b}_{X,u}$, if $\operatorname{supp} b\subseteq B(x_0,r)$ for some $x_0\in \mathbb{R}^n, r>0$, and
	$$\|b\|_X \leq \frac{1}{u(x_0,r)}.$$
	The block space $\mathfrak{B}_{X,u}$ associated with $X$ is defined by
	$$\mathfrak{B}_{X,u} =\left\{\sum_{k=1}^{\infty} \lambda_k b_k: \sum_{k=1}^{\infty} |\lambda_k|<\infty \ \text{and} \ b_k\in \mathfrak{b}_{X,u}\right\}$$
	and endowed with the norm 
	$$\|f\|_{\mathfrak{B}_{X,u}} = \inf \left\{\sum_{k=1}^\infty |\lambda_k|: f=\sum_{k=1}^\infty \lambda_k b_k \ \text{a.e.}\right\}.$$
\end{definition}
 For the fundamental properties of the block spaces, refer to \cite{ho4}. Similarly to Theorem \ref{holder-X}, the H\"older inequality for $M_X^u$ and $\mathfrak{B}_{X',u}$ is established in the following theorem, see \cite[Lemma 4.2]{ho4} and \cite[Propsition 2.5]{ho3} for the proof details 
\begin{lemma}[\!\!\cite{ho4,ho3}]
	Let $X$ be a Banach space and $u: \mathbb{R}^n \times (0,\infty) \to (0,\infty)$ be a Lebesgue measurable function. Then for any $f\in M_X^u$ and $g\in \mathfrak{B}_{X',u}$, we have
	\begin{align}\label{holder-M-B}
		\int_{\mathbb{R}^n} |f(x)g(x)|dx \leq C\|f\|_{M_X^u} \|g\|_{\mathfrak{B}_{X',u}}.
	\end{align}
	In addition, there exist constants $C_0, C_1>0$ such that for any $f\in M_X^u$,
	\begin{align}\label{mxu<sup<mxu}
		C_0\|f\|_{M_X^u} \leq \sup_{b\in \mathfrak{b}_{X',u}} \int_{\mathbb{R}^n} |f(x)b(x)| dx \leq C_1\|f\|_{M_X^u}.
	\end{align}
\end{lemma}
\noindent The following results concern the boundedness of the Hardy-Littlewood maximal operator $M$ on the block space $\mathfrak{B}_{X',u}$ (\!\!\cite[Theorem 4.4]{ho4}) and the multilinear maximal function $\mathscr{M}$ on the product Morrey-Banach space (\!\!\cite[Lemma 3.3]{zha}).
\begin{lemma}[\!\!\cite{ho4}]\label{the boundedness of the Hardy-Littlewood maximal operator on {B}_{X',u}}
	Let $X$ be a Banach space and $u: \mathbb{R}^n \times (0,\infty) \to (0,\infty)$ be a Lebesgue measurable function. If $X'\in \mathbb{M}$ and $u\in \mathbb{W}_X^0$, then the Hardy-Littlewood maximal operator $M$ is bounded on $\mathfrak{B}_{X',u}$.
\end{lemma}
\begin{lemma}[\!\!\cite{zha}]\label{thm-multi-maximal-bounded-mxu}
	Let $X, X_i, i=1,\dots,m$ are Banach function spaces with $\|\chi_B\|_X \leq C\prod_{i=1}^m$ $\|\chi_B\|_{X_i}$ for any ball $B\subseteq \mathbb{R}^n$ and $X_i\in \mathbb{M}\cup \mathbb{M}'$. If $\mathscr{M}$ is bounded from $X_1\times\cdots\times X_m$ to $X$, then $\mathscr{M}$ is bounded from $M_{X_1}^{u_1} \times \cdots \times M_{X_m}^{u_m}$ to $M_{X}^{u}$ with $u=\prod_{i=1}^m u_i$ and $u_i\in \mathbb{W}_{X_i}^1, i=1,\dots,m$.
\end{lemma}
\begin{remark}\label{wxi0}
	Upon carefully reviewing the proof of Theorem \ref{thm-multi-maximal-bounded-mxu} in \cite{zha}, we find that the result remains valid for $u_i\in \mathbb{W}_{X_i}^0, i=1,\dots,m$. Moreover, Theorem \ref{thm-multi-maximal-bounded-mxu} also holds in the case $m=1$,  implying that the Hardy-Littlewood maximal operator $M$ is bounded from $M_X^u$ into $M_X^u$ with $u\in \mathbb{W}_{X}^{0}$.
\end{remark}

Now we present the extrapolation theory \cite[Theorem 3.1]{ho4} for Morrey-Banach spaces.
\begin{lemma}[\!\!\cite{ho4}]\label{ho4-thm3.1}
	Let $X$ be Banach function space, $0<p_0<\infty$ and $f,g\in \mathcal{M}(\mathbb{R}^n)$ be non-negative functions that be not identically zero. Suppose that for every $\omega\in A_1$, we have 
	\begin{align}\label{ho4-thm3.1-eq1}
		\int_{\mathbb{R}^{n}} [f(x)]^{p_0} w(x) dx \leq C\int_{\mathbb{R}^{n}} [g(x)]^{p_0} w(x) dx,
	\end{align}
	where $C$ is independent of $f,g$ and $p_0$.
	Suppose that there exist $q_0$ with $p_0\leq q_0<\infty$ and $C>0$, such that $X^\frac{1}{q_0}$ is a Banach function space, $X^\frac{1}{q_0}\in \mathbb{M}'$ and
	\begin{gather*}
		u(x,2r) \leq Cu(x,r),\quad
		\sum_{j=0}^\infty \frac{\|\chi_{B(x,r)}\|_{X^{{1}/{q_0}}}}{\|\chi_{B(x,2^{j+1}r)}\|_{X^{1/{q_0}}}} (u(x,2^{j+1}r))^{q_0} < C (u(x,r))^{q_0}
	\end{gather*}
	for all $x\in \mathbb{R}^n, r>0$. Then for non-negative functions $f,g\in \mathcal{M}(\mathbb{R}^n)$ satisfying that $f,g \not\equiv 0$,
	$$\|f\|_{M_X^u} \leq C\|g\|_{M_X^u}.$$
	Moreover, for every $q\in (1,\infty)$ and non-negative functions $f_i,g_i\in \mathcal{M}(\mathbb{R}^n)$ satisfying \eqref{ho4-thm3.1-eq1} and that $f_i,g_i \not\equiv 0$, $i\in \mathbb{N}$, we have
	\begin{align*}
		\left\| \left(\sum_{i\in\mathbb{N}} |f_i|^q\right)^\frac{1}{q}  \right\|_{M_X^u} \leq C \left\| \left(\sum_{i\in\mathbb{N}} |g_i|^q\right)^\frac{1}{q}  \right\|_{M_X^u}.
	\end{align*}
\end{lemma}

\subsection{Lorentz space}
\ \\
\indent This subsection aims to present the definition of the Lorentz space and  its related properties, more details can be found in \cite{gra}.
\begin{definition}[Lorentz space]
	Given $f\in \mathcal{M}(\mathbb{R}^n,u)$ and $0<p,q\leq\infty$, define
	$$\|f\|_{L^{p,q}(u)} = \left\{
	\begin{aligned}
		&\left(\int_0^\infty \left(t^\frac{1}{p} f^\ast(t)\right)^q\frac{dt}{t}\right)^\frac{1}{q}, &\text{if}\ q<\infty,\\
		&\sup_{t>0} t^\frac{1}{p} f^\ast(t), &\text{if}\ q=\infty,
	\end{aligned}\right.$$
	where $f^\ast$ is the rearrangement function of $f$, precisely, $f^\ast(t)=\inf\{s>0: d_f(s)\leq t\}$ with $d_f(s)=u\left(\{x\in \mathbb{R}^n: |f(x)|>s\}\right)$. The set of all $f$ with $\|f\|_{L^{p,q}(u)}<\infty$ is denoted by $L^{p,q}(u)$ and is called the Lorentz space with indices $p$ and $q$. In particular, the notation $L^{p,q}(\rn)$ is reserved for $L^{p,q}(|\cdot|)$.
\end{definition}

	We highlight the following three key points:
	\begin{enumerate}
		\item For $0<p=q\leq\infty$, the space $L^{p,q}(\mathbb{R}^n)$coincides with the classical Lebesgue space $L^p(\mathbb{R}^n)$. Moreover, for $1<p<\infty$ and $1\leq q\leq \infty$, the space $L^{p,q}(\mathbb{R}^n)$ forms a Banach function space. 
		\item According to the definition of Lorentz spaces, if $p>0, 0<q<t\leq \infty$, then the embedding $L^{p,q}(\rn) \subset L^{p,t}(\rn)$ holds.
	 Consequently, for $q<p$, we obtain $L^{p,q}(\rn) \subset L^{p,p}(\rn) = L^{p}(\rn)$. In particular, the inclusion relationships $L^{p,1}(\rn) \subseteq L^{p}(\rn) \subset L^{p,\infty}(\rn)$ hold for $p\geq 1$.	
		\item As stated in \cite[Propsition 1.4.16]{gra}, the dual space of $L^{p,q}(\mathbb{R}^n)$  is given by $L^{p',q'}(\mathbb{R}^n)$ when $1<p,q<\infty$. Additionally, for  $1<p<\infty$, the dual of  $L^{p,1}(\mathbb{R}^n)$ is the space $L^{p',\infty}$.
	\end{enumerate}

The H\"older inequality in Lorentz spaces is a useful tool. We state it in the following lemma and provide a detailed proof.
\begin{lemma}\label{lem4}
	Let $1<p_i<\infty, 0<q_i\leq \infty$ or $p_i=q_i=\infty$, $i=1,\dots,m$. If $f_i\in L^{p_i,q_i}(\mathbb{R}^n), i=1,\dots,m$, then $\prod_{i=1}^m f_i\in L^{p,q}(\mathbb{R}^n)$. Moreover,
	\begin{subnumcases}
		{\bigg\|\prod_{i=1}^m f_i\bigg\|_{L^{p,q}(\mathbb{R}^n)} \leq }
		m^\frac{1}{p} \prod_{i=1}^m\|f_i\|_{L^{p_i,q_i}(\mathbb{R}^n)}, &if $0<q<\infty$,\label{eq1}\\
		p^{-\frac{1}{p}} \prod_{i=1}^m p_i^\frac{1}{p_i}\|f_i\|_{L^{p_i,q_i}(\mathbb{R}^n)}, &if $p<\infty, q=\infty$,\label{eq2}\\
		\prod_{i=1}^m\|f_i\|_{L^{p_i,q_i}(\mathbb{R}^n)}, &if $p=q=\infty$\label{eq3},
	\end{subnumcases}
	where $\frac{1}{p}=\sum_{i=1}^m \frac{1}{p_i}$ and $\frac{1}{q}=\sum_{i=1}^m\frac{1}{q_i}$.
\end{lemma}
\begin{proof}
	Consider first the case $p = q = \infty$. For $i=1,\dots,m$, $L^{p_i,q_i}(\mathbb{R}^n) = L^{p,q}(\mathbb{R}^n) = L^\infty (\mathbb{R}^n)$, which shows that \eqref{eq3} is valid directly.
	
	Next consider the case $\frac{1}{m}<p<\infty, q=\infty$. Note that $q=\infty$ means that $q_1=\cdots=q_m=\infty$, then \eqref{eq2} holds by \cite[Exercise 1.1.15]{gra}.
	
	Finally, consider the case $\frac{1}{m}<p<\infty, 0<q<\infty$. From \cite{dao}, we further discuss this in two categories. \\
	(1) For any $i\in \{1,\dots,m\}$, $q_i\neq \infty$. According to \cite[Proposition 1.4.5]{gra}, 
	\begin{align*}
		\left(\prod_{i=1}^m f_i\right)^\ast \left(\sum_{i=1}^m t_i\right) \leq \prod_{i=1}^m f_i^\ast(t_i)
	\end{align*}
	for any $t_i>0$. H\"older inequality asserts that 
	\begin{align*}
		\left\|\prod_{i=1}^m f_i\right\|_{L^{p,q}(\mathbb{R}^n)}^q &= \int_0^\infty \left(t^\frac{1}{p} \left(\prod_{i=1}^m f_i\right)^\ast(t)\right)^q \frac{dt}{t}\\
		&=m^\frac{q}{p} \int_0^\infty \left(t^\frac{1}{p} \left(\prod_{i=1}^m f_i\right)^\ast(mt)\right)^q \frac{dt}{t}\\
		&\leq m^\frac{q}{p} \int_0^\infty \prod_{i=1}^m \left(t^\frac{1}{p_i} f_i^\ast(t)\right)^q \frac{dt}{t}\\
		&\leq m^\frac{q}{p} \prod_{i=1}^m \left(\int_0^\infty \left(t^\frac{1}{p_i} f_i^\ast(t)\right)^{q_i} \frac{dt}{t}\right)^\frac{q}{q_i}\\
		&= m^\frac{q}{p} \prod_{i=1}^m \|f_i\|_{L^{p_i,q_i}(\mathbb{R}^n)}^q.
	\end{align*}
	(2) When there exists a $j\in\{1,\dots,m\}$ such that $q_j=\infty$, suppose that $q_i<\infty$ for any $i\neq j$. Then,
	\begin{align*}
		\left\|\prod_{i=1}^m f_i\right\|_{L^{p,q}(\mathbb{R}^n)}^q &\leq m^\frac{q}{p} \int_0^\infty \prod_{i=1}^m \left(t^\frac{1}{p_i} f_i^\ast(t)\right)^q \frac{dt}{t}\\
		&\leq m^\frac{q}{p} \left(\sup_{t>0} t^{\frac{1}{p_j}} f_j^\ast(t)\right)^q \int_0^\infty \left(\prod_{i\neq j}t^\frac{1}{p_i} f_i^\ast(t)\right)^{q} \frac{dt}{t}\\
		&\leq m^\frac{q}{p} \|f_i\|_{L^{p_j,\infty}}^q \prod_{i=1}^m \|f_i\|_{L^{p_i,q_i}(\mathbb{R}^n)}^q,
	\end{align*}
	which completes the proof. 
\end{proof}

The following lemma concerns the weighted boundedness of the Hardy-Littlewood maximal operator on Lorentz space.
\begin{lemma}[\!\!\cite{acc}]\label{acc-cor2.2}
	Let $1<p<\infty, 0<q\leq \infty$ and $u\in A_p$ with $\sigma\coloneqq u^{1-p'}$. Then there exists a constant $C_{q,n}$ depending on $q$ and $n$ such that 
	\begin{align}\label{2}
		\|M\|_{L^{p,q}(u)} \leq C_{q,n} (p')^{\frac{1}{\min\{p,q\}}} [u]_{A_p}^{\frac{1}{p}} [\sigma]_{A_\infty}^{\frac{1}{\min\{p,q\}}}.
	\end{align}
\end{lemma}

The exponents of the weight in the right-hand side of \eqref{2} are optimal, in the sense that, if any of the exponents is replaced by a smaller one, the result no longer holds.

\subsection{Young function and Orlicz space}
\ \\
\indent  A function $\Phi$ defined on $[0,\infty)$ is called a Young function if it is a continuous increasing convex function satisfying
$$\Phi(0)=0 \quad \text{and} \quad \lim_{t\to\infty} \Phi(t) =\infty.$$

As another generalization of Lebesgue spaces, the Orlicz space was introduced by Orlicz in \cite{orl}.
\begin{definition}[\!\!\cite{orl}]
	Given a Young function $\Phi$, the set
	$$L_{\Phi}(\mu) = \left\{f\in L_{\text{loc}}^{1}(\mu): \int_{\rn} \Phi\left(k|f(x)|\right)d\mu <\infty \; \text{for some}\; k>0\right\}$$
	is called Orlicz space. Moreover, for any $f\in L_{\Phi}(\mu)$, its Luxemburg norm is defined by
	$$\|f\|_{\Phi(\mu)} \coloneqq \inf \left\{ \lambda>0: \int_{\rn} \Phi\left(\frac{|f(x)|}{\lambda}\right) d\mu \leq 1\right\}.$$
\end{definition}
Given a fixed cube $Q$, the localized version of the Orlicz space, denoted by $L_{\Phi}\left(Q, \frac{\mu}{\mu(Q)}\right)$ consists of all functions $f\in L_{\text{loc}}^{1}(\mu)$ for which there exists a constant $\lambda>0$ such that
$$\frac{1}{\mu(Q)} \int_{Q} \Phi\left(k|f(x)|\right)d\mu <\infty, \quad \hbox{for some } k>0.$$
The corresponding Luxemburg norm is given by
$$\|f\|_{L_{\Phi}(Q,\mu/\mu(Q))} \coloneqq \|f\|_{\Phi(\mu),Q} \coloneqq \inf \left\{ \lambda>0: \frac{1}{|Q|} \int_{Q} \Phi\left(\frac{|f(x)|}{\lambda}\right) d\mu \leq 1\right\}.$$

For convenience, when \( \mu \) is the Lebesgue measure, we denote the norm simply as \( \|f\|_{\Phi,Q} \).
In particular, for some \( a \geq 0 \) and specific choices of \( \Phi \), the notation can be specialized as follows:
\begin{itemize}
	\item When \( \Phi(t) = t^p \log^a (e+t) \), we write \( \|f\|_{L^p(\log L)^a, Q} \).
	\item When \( \Phi(t) = \exp(t^a) - 1 \), we use \( \|f\|_{\exp(L^a), Q} \).
\end{itemize}

\subsection{Sparse family}\label{spare}
\ \\
\indent 
Given a cube \( Q_0 \subset \mathbb{R}^n \), let \( \mathcal{D}(Q_0) \) denote the collection of all dyadic subcubes of \( Q_0 \), i.e., the cubes obtained by iteratively subdividing \( Q_0 \) and each of its descendants into \( 2^n \) congruent subcubes.

\begin{definition}[\!\!\cite{ler6}]
	A dyadic lattice $\mathscr{D}$ in $\mathbb{R}^n$ is any collection of cubes such that 
	\begin{enumerate}
			\item if $Q\in \mathscr{D}$, then each descendant of $Q$ is in $\mathscr{D}$ as well;
			\item every two cubes $Q', Q'' \in \mathscr{D}$ have a common ancestor, that is, there exists $Q\in \mathscr{D}$ such that $Q', Q''\in \mathcal{D}(Q)$;
			\item for every compact set $K\subset \mathbb{R}^n$, there exists a cube $Q\in \mathscr{D}$ containing $K$. 
		\end{enumerate}
\end{definition}
In \cite[Theorem 3.1]{ler6}, Lerner and Nazarov established the following Three Lattice Theorem, which offers a clear understanding of the structure of dyadic lattices.
\begin{lemma}[\!\!\cite{ler6}]
	For every dyadic lattice $\mathcal{D}$, there exist $3^n$ dyadic lattices $\{\mathcal{D}_j\}_{j=1}^{3^n}$  such that
	$$	\{3 Q: Q \in \mathcal{D}\}=\bigcup_{j=1}^{3^n} \mathcal{D}_j	$$
	and for each cube $Q \in \mathcal{D}$, there is a unique cube $R_Q \in \mathcal{D}_{j}$ for some $j$ such that $Q \subseteq R_Q$ and $3 l_Q=l_{R_Q}$.
\end{lemma}

For a dyadic lattice $\mathscr{D}$, a family $\mathcal{S} \subset \mathscr{D}$ is called $\eta$-sparse, $0<\eta<1$, if for any cube $Q\in \mathcal{S}$,
$$\left|\bigcup_{Q'\in \mathcal{S}: Q'\subsetneq Q} Q' \right| \leq (1-\eta) |Q|.$$
In particular, if $\mathcal{S} \subset \mathscr{D}$ is $\eta$-sparse, then for any $Q\in \mathcal{S}$, we define
$$E_Q = Q\backslash \bigcup_{Q'\in \mathcal{S}: Q'\subsetneq Q} Q'.$$
It follows that  $|E_Q| \geq \eta |Q|$ and the sets $\{E_Q\}_{Q\in \mathcal{S}}$ are pairwise disjoint.

Let $\mathcal{D}$ be a dyadic lattice and $\mathcal{S}\subseteq \mathcal{D}$ be a $\eta$-sparse family, the sparse operator $\mathcal{T}_{r, \mathcal{S}}$ with $r>0$ is defined by
$$\mathcal{T}_{r, \mathcal{S}} f(x)= \sum_{Q \in \mathcal{S}}\langle|f|^{r}\rangle_{Q}^{1 / r}\chi_{Q}(x)=\sum_{Q \in \mathcal{S}}\left(\frac{1}{|Q|} \int_{Q}|f(y)|^{r} d y\right)^{\frac{1}{r}} \chi_{Q}(x).$$

\vspace{0.1cm}

\section{Proofs of Theorems \ref{thm1} and \ref{thm2}}\label{pr+1+2}
This section is dedicated to proving Theorems \ref{thm1} and \ref{thm2}. To prove Theorem \ref{thm1}, we first invoke the following lemma from \cite[Theorem 1.8]{bp+tan}.
\begin{lemma}[\!\!\cite{bp+tan}]\label{yinli3}
	Let $I\coloneqq\{i_1,\ldots,i_l\}=\{1,\ldots,l\}\subseteq \{1,\ldots,m\}.$  If $\vec{b} \in \mathrm{BMO}^l$ and $\mathcal{T}_{\vec{b}}$ satisfies the Hypothesis \ref{hyp}, then for any $0<p<\infty, w\in A_\infty$,
	\begin{equation}
		\int_{\mathbb{R}^n}\big|\mathcal{T}_{\vec{b}}(\vec{f})(x)\big|^pw(x)dx \lesssim \|\vec{b}\|_{\mathrm{BMO}}^p [w]_{A_\infty}^{pl+\max\{2,p\}}
		\int_{\mathbb{R}^n}\left|\mathscr{M}_{L(\log L)}(\vec{f})(x)\right|^pw(x)dx.
	\end{equation}
\end{lemma}

\begin{proof}[\bf Proof of Theorem \ref{thm1}]
	Since \( X' \in \mathbb{M} \) and \( u \in \mathbb{W}_X^0 \), it follows from Theorem \ref{the boundedness of the Hardy-Littlewood maximal operator on {B}_{X',u}} that the Hardy-Littlewood maximal operator \( M \) is bounded on \( \mathfrak{B}_{X',u} \).
		Let \( \|M\|_{\mathfrak{B}_{X',u}} \) denote the operator norm of \( M: \mathfrak{B}_{X',u} \to \mathfrak{B}_{X',u} \).
		Next, we recall an operator from the Rubio de Francia algorithm \cite{rub} in the context of extrapolation theory. For any locally integrable function \( h \), define
	\[
	\mathcal{R}h(x) \coloneqq \sum_{k=0}^\infty \frac{M^k h(x)}{2^k \|M^k\|_{\mathfrak{B}_{X',u}}},
	\]
	where \( M^0 h \coloneqq |h| \) and \( M^k \) denotes the \( k \)-th iteration of \( M \).
		Whenever \( M: \mathfrak{B}_{X',u} \to \mathfrak{B}_{X',u} \) is bounded, the operator \( \mathcal{R}h \) is well defined. By its definition and Theorem \ref{the boundedness of the Hardy-Littlewood maximal operator on {B}_{X',u}}, \( \mathcal{R} \) satisfies the following properties.
	\begin{gather}
		0\leq |h(x)| \leq \mathcal{R}h(x),\label{h<Rh}\\
		\|\mathcal{R}h\|_{\mathfrak{B}_{X',u}} \leq 2\|h\|_{\mathfrak{B}_{X',u}},\label{Rh<h}\\
		\mathcal{R}h\in A_1 \  \text{and} \  [\mathcal{R}h]_{A_1} \leq 2\|M\|_{\mathfrak{B}_{X',u}}.\label{rh-a1}
	\end{gather}
	For any $h\in \mathfrak{b}_{X',u}$, take $h_1=h$, $\lambda_1=1$ and $\lambda_k=0$ for $k\geq 2$, we have $h=\sum_{k=1}^\infty \lambda_k h_k \in \mathfrak{B}_{X',u}$ and 
	\begin{equation}
	\begin{aligned}\label{h<1}
		\|h\|_{\mathfrak{B}_{X',u}} 
		&= \inf\left\{\sum_{k=1}^\infty |\lambda_k|: ~h=\sum_{k=1}^\infty \lambda_k h_k \ \text{a.e.}\right\} \\
		&\leq \inf\left\{\sum_{k=1}^\infty |\lambda_k|: ~h=\sum_{k=1}^\infty \lambda_k h_k \ \text{with} \ \lambda_1=1, h_1=h \ \text{and} \ \lambda_k=0\ \text{for} \ k\geq 2\right\} \\
		&\leq 1.
	\end{aligned}
	\end{equation}
	Moreover, \eqref{holder-M-B}, \eqref{Rh<h} and \eqref{h<1} guarantee that
	\begin{equation}
		\begin{aligned}\label{m-h-l}
		\int_{\mathbb{R}^{n}} \left|\mathscr{M}_{L(\log L)} (\vec{f})(x)\right|^p \mathcal{R}h(x) dx &\leq C \left\|\left(\mathscr{M}_{L(\log L)} (\vec{f})\right)^p\right\|_{M_X^u} \|\mathcal{R}h\|_{\mathfrak{B}_{X',u}}\\
		&\leq C \left\|\left(\mathscr{M}_{L(\log L)} (\vec{f})\right)^p\right\|_{M_X^u} \|h\|_{\mathfrak{B}_{X',u}}\\
		&\leq C \left\|\left(\mathscr{M}_{L(\log L)} (\vec{f})\right)^p\right\|_{M_X^u}.
	\end{aligned}
	\end{equation}
	Therefore, if $\left(\mathscr{M}_{L(\log L)} (\vec{f})\right)^p \in M_X^u$, then we have $\left(\mathscr{M}_{L(\log L)} (\vec{f})\right)^p \in \bigcap_{h\in \mathfrak{b}_{X',u}} L^1(\mathcal{R}h)$. 
	
	Note that \eqref{rh-a1}, Lemma \ref{yinli3} and the fact $\mathcal{R}h\in A_1\subseteq A_\infty$ assure that for any $h\in \mathfrak{b}_{X',u}$,
	\begin{align*}
		\int_{\mathbb{R}^{n}} \left|\mathcal{T}_{\vec{b}}(\vec{f})(x) \right|^p \mathcal{R}h(x) dx \lesssim \|\vec{b}\|_{\text{BMO}}^p \|M\|_{\mathfrak{B}_{X',u}}^{pl+\max\{2,p\}} \int_{\mathbb{R}^n} \left(\mathscr{M}_{L(\log L)} (\vec{f})(x)\right)^p \mathcal{R}h(x)dx
	\end{align*}
	holds for those $\vec{f}\in \mathcal{M}^m$ such that $\left(\mathscr{M}_{L(\log L)} (\vec{f})\right)^p \in M_X^u$. Thus, applying \eqref{h<Rh}, one may get
	\begin{align*}
		\int_{\mathbb{R}^{n}} \left|\mathcal{T}_{\vec{b}}(\vec{f})(x) \right|^p |h(x)| dx
		&\leq \int_{\mathbb{R}^{n}} \left(\mathcal{T}_{\vec{b}}(\vec{f})(x) \right)^p \mathcal{R}h(x) dx \\
		&\lesssim \|\vec{b}\|_{\text{BMO}}^p \|M\|_{\mathfrak{B}_{X',u}}^{pl+\max\{2,p\}} \int_{\mathbb{R}^n} \left(\mathscr{M}_{L(\log L)} (\vec{f})(x)\right)^p \mathcal{R}h(x)dx.
	\end{align*}
	Combined with \eqref{m-h-l}, we obtain
	\begin{align*}
		\int_{\mathbb{R}^{n}} \left|\mathcal{T}_{\vec{b}}(\vec{f})(x) \right|^p |h(x)| dx
		&\lesssim \|\vec{b}\|_{\text{BMO}}^p \|M\|_{\mathfrak{B}_{X',u}}^{pl+\max\{2,p\}} \left\|\left(\mathscr{M}_{L(\log L)} (\vec{f})\right)^p\right\|_{M_X^u}.
	\end{align*}
	Therefore, \eqref{mxu<sup<mxu} gives that
	\begin{align*}
		\left\|\left(\mathcal{T}_{\vec{b}}(\vec{f})\right)^p \right\|_{M_X^u}
		&\lesssim \sup_{h\in \mathfrak{b}_{X',u}} \int_{\mathbb{R}^n} \left|\mathcal{T}_{\vec{b}}(\vec{f})\right|^p |h(x)|dx\\
		&\lesssim \|\vec{b}\|_{\text{BMO}}^p \|M\|_{\mathfrak{B}_{X',u}}^{pl+\max\{2,p\}} \left\|\left(\mathscr{M}_{L(\log L)} (\vec{f})\right)^p\right\|_{M_X^u}.
	\end{align*}
\end{proof}

The remainder of this section is dedicated to demonstrating Theorem \ref{thm2}.
\begin{proof}[\bf Proof of Theorem \ref{thm2}]
	Note that for any Banach function space \( X \), the space \( X^p \) (for \( p \geq 1 \)) is also a Banach function space. Indeed, for \( X^p \) with \( p \geq 1 \), it is straightforward to verify that conditions (1) and (2) in Definition \ref{def-bfs} hold. 
	
	If \( 0 \leq f_n \uparrow f \) almost everywhere, then \( |f_n|^p \uparrow |f|^p \). Moreover, if \( |f_n|^p, |f|^p \in X \), then  
	\[
	\||f_n|^p\|_X \uparrow \||f|^p\|_X.
	\]  
	Thus, we obtain  
	\[
	\|f_n\|_{X^p} = \||f_n|^p\|_X^{\frac{1}{p}} \uparrow \||f|^p\|_X^{\frac{1}{p}} = \|f\|_{X^p},
	\]  
	which establishes condition (3).  
	
	For condition (4), observe that \( \chi_E \in X \) and  
	\[
	\|\chi_E\|_{X^p} = \|\chi_E\|_X^{\frac{1}{p}} < \infty,
	\]  
	implying that \( \chi_E \in X^p \).  
	
	Regarding condition (5), for any \( f \in X^p \), the case \( p = 1 \) is trivial. If \( p > 1 \), H\"older inequality gives  
	\[
	\int_E |f(x)|dx \leq \left(\int_E |f(x)|^p dx\right)^\frac{1}{p} |E|^\frac{1}{p'}.
	\]  
	Since \( |f|^p \in X \), there exists a constant \( C_E \) such that  
	\[
	\int_E |f(x)|^p dx \leq C_E \||f|^p\|_X.
	\]  
	Thus, we obtain  
	\[
	\int_E |f(x)|dx \leq C_E^\frac{1}{p} |E|^\frac{1}{p'} \|f\|_{X^p},
	\]  
	which completes the proof.
	
	Next we turn to prove that if $u^\frac{1}{p} \in \mathbb{W}_{X^p}^0$, then $u\in \mathbb{W}_X^0$. According to Definition \ref{def-wxalpha}, for any Lebesgue measurable function $u^\frac{1}{p} \in \mathbb{W}_{X^p}^0$ satisfying $u(x_1,r_1)\leq u(x_2,r_2)$, there exists a constant $C>0$ such that 
	\begin{align*}
		\frac{\|\chi_{B(x_1,r_1)}\|_{X}}{u(x_1,r_1)} 
		&= \left(\frac{\|\chi_{B(x_1,r_1)}\|_{X}^\frac{1}{p}}{u^\frac{1}{p} (x_1,r_1)}\right)^p 
		= \left(\frac{\|\chi_{B(x_1,r_1)}\|_{X^p}}{u^\frac{1}{p} (x_1,r_1)}\right)^p\\
		&\leq C \left(\frac{\|\chi_{B(x_2,r_2)}\|_{X^p}}{u^\frac{1}{p} (x_2,r_2)}\right)^p
		=\frac{\|\chi_{B(x_2,r_2)}\|_{X}}{u(x_2,r_2)},
	\end{align*}
	and 
	\begin{align*}
		\sum_{j=0}^{\infty} \frac{\|\chi_{B(x,r)}\|_{X}}{\|\chi_{B(x,2^{j+1}r)}\|_{X}} u(x,2^{j+1}r) 
		&= \sum_{j=0}^{\infty} \left[\frac{\|\chi_{B(x,r)}\|_{X^p}}{\|\chi_{B(x,2^{j+1}r)}\|_{X^p}} u^\frac{1}{p}(x,2^{j+1}r)\right]^p\\
		&\leq \left(\sum_{j=0}^{\infty} \frac{\|\chi_{B(x,r)}\|_{X^p}}{\|\chi_{B(x,2^{j+1}r)}\|_{X^p}} u^\frac{1}{p}(x,2^{j+1}r)\right)^p\\
		&\leq C u(x,r),
	\end{align*}
	which guarantees that $u\in \mathbb{W}_X^0$. Notice that $X'\in \mathbb{M}$, Theorem \ref{thm1} tells us 
	\begin{align}\label{eq-1}
		\left\| (\mathcal{T}_{\vec{b}} (\vec{f}))^p \right\|_{M_{X}^{u}} \lesssim \|\vec{b}\|_{\text{BMO}}^p \left\|M\right\|_{\mathfrak{B}_{X',u}}^{pl+\max\{2,p\}} \left\|\left(\mathscr{M}_{L(\log L)} (\vec{f}) \right)^p \right\|_{M_{X}^u}.
	\end{align}
A straightforward calculation shows that  $$\mathscr{M}_{L(\log L)} (\vec{f})(x) \leq C_n \mathscr{M}\left(Mf_1,\dots,Mf_m\right)(x)$$ and $\|g\|_{M_{X^p}^{u^{{1}/{p}}}}^p = \||g|^p\|_{M_X^u}$ for any $g\in M_{X^p}^{u^{{1}/{p}}}$. Invoking \eqref{eq-1} and Theorem \ref{thm-multi-maximal-bounded-mxu}, we obtain 
	\begin{align}\label{thm2-eq3}
		\left\| (\mathcal{T}_{\vec{b}} (\vec{f}))^p \right\|_{M_{X}^{u}} 
		\lesssim \|\vec{b}\|_{\text{BMO}}^p 
		\left\|M\right\|_{\mathfrak{B}_{X',u}}^{pl+\max\{2,p\}} 
		\left\|\mathscr{M}\right\|_{M_{X_1^p}^{u_1^{1/p}} \times \cdots \times M_{X_m^p}^{u_m^{1/p}} \to M_{X^p}^{u^{1/p}}}^p 
		\prod_{i=1}^m \|Mf_i\|_{M_{X_i^p}^{u_i^{1/p}}}^p.
	\end{align}
	According to Remark \ref{wxi0}, it suffices to deal with $\|Mf_i\|_{M_{X_i^p}^{u_i^{1/p}}}$ for $i=1,\dots,m$ in the following two cases.\\
	(1) For $p=1$ and $X_i \in \mathbb{M}, i=1,\dots,m$, utilize Theorem \ref{thm-multi-maximal-bounded-mxu} with $m=1$ and $u_i\in \mathbb{W}_{X_i}^0$ to estimate
	$$\|Mf_i\|_{M_{X_i}^{u_i}} \leq \|M\|_{M_{X_i}^{u_i} \to M_{X_i}^{u_i}} \|f_i\|_{M_{X_i}^{u_i}}.$$
	Thus, by \eqref{thm2-eq3}, we have
	$$\left\| \mathcal{T}_{\vec{b}} (\vec{f}) \right\|_{M_{X}^{u}}
	\lesssim \|\vec{b}\|_{\text{BMO}} 
	\left\|M\right\|_{\mathfrak{B}_{X',u}}^{l+2} 
	\left\|\mathscr{M}\right\|_{M_{X_1}^{u_1} \times \cdots \times M_{X_m}^{u_m} \to M_{X}^{u}}
	\prod_{i=1}^m \|M\|_{M_{X_i}^{u_i} \to M_{X_i}^{u_i}} \|f_i\|_{M_{X_i}^{u_i}}.$$
	(2) For $p>1$, notice that $X_i^p, i=1,\dots,m$, are Banach function spaces and $u_i^{\frac{1}{p}} \in \mathbb{W}_{X_i^p}^0$. Thus, $(X_i^p)^\frac{1}{p} = X_i \in \mathbb{M}'$ and $u_i\in \mathbb{W}_{X_i}^0$. Using Lemma \ref{ho4-thm3.1} with $f=Mf_{i}, g=f_{i}$ and $p_0=q_0=p>1$, we obtain
	\begin{align*}
		\|Mf_i\|_{M_{X_i^p}^{u_i^{1/p}}} \leq C\|f_i\|_{M_{X_i^p}^{u_i^{1/p}}} = C\||f_i|^p\|_{M_{X_i}^{u_i}}^\frac{1}{p}.
	\end{align*}
	Therefore, by \eqref{thm2-eq3}, we have
	$$\left\| (\mathcal{T}_{\vec{b}} (\vec{f}))^p \right\|_{M_{X}^{u}}
	\lesssim \|\vec{b}\|_{\text{BMO}}^p
	\left\|M\right\|_{\mathfrak{B}_{X',u}}^{pl+\max\{2,p\}} 
	\left\|\mathscr{M}\right\|_{M_{X_1^p}^{u_1^{1/p}} \times \cdots \times M_{X_m^p}^{u_m^{1/p}} \to M_{X^p}^{u^{1/p}}}^p
	\prod_{i=1}^m \||f_i|^p\|_{M_{X_i}^{u_i}}.$$
	Taking (1) and (2) into account completes the proof of Theorem \ref{thm2}.
\end{proof}

Unlike the previous proof, in what follows, we do not rely on Rubio de Francia extrapolation. Instead, we estimate the sparse operators directly to provide an alternative proof of Theorem \ref{thm2}.
\begin{proof}[Another proof of Theorem \ref{thm2}]
	Based on \eqref{mxu<sup<mxu} and Hypothesis \ref{hyp}, we have 
	\begin{equation}
		\begin{aligned}\label{eq2-6}
		\left\|\mathcal{T}_{\vec{b}}(\vec{f})\right\|_{M_{X}^{u}} &\leq \frac{1}{C_{0}} \sup_{h\in \mathfrak{b}_{X',u}} \int_{\rn} \left|\mathcal{T}_{\vec{b}}(\vec{f})(x) h(x)\right| dx\\
		&\lesssim \sum_{j=1}^{3^n} \sum_{\vec{\gamma} \in \{1,2\}^l} \sup_{h\in \mathfrak{b}_{X',u}} \int_{\rn}  \mathcal{T}_{\mathcal{S}_j,\vec{b}}^{\vec{\gamma}} (\vec{f})(x) |h(x)| dx\\
		&\eqqcolon \sum_{j=1}^{3^n} \sum_{\vec{\gamma} \in \{1,2\}^l} \mathcal{H}_{\vec{\gamma},j}(\vec{f}),
	\end{aligned}
	\end{equation}
	where $$\mathcal{H}_{\vec{\gamma},j}(\vec{f}) = \sup_{h\in \mathfrak{b}_{X',u}} \int_{\rn}  \mathcal{T}_{\mathcal{S}_j,\vec{b}}^{\vec{\gamma}} (\vec{f})(x) |h(x)| dx.$$
	Fixed $j\in \{1,\dots,3^{n}\}$, without loss of generality, let $\vec{\gamma} = (\overbrace{ 1,\dots,1 }^{l_1}, \overbrace{ 2,\dots,2 }^{l-l_1})$ with $0\leq l_{1} \leq l$. We divide the proof into three parts.
	
	\noindent (1) For $l_{1}=0$, we have $\vec{\gamma}=(2,\dots,2)$. Thus, 
	\begin{equation}
		\begin{aligned}\label{eq2-1}
		\mathcal{H}_{\vec{2},j}(\vec{f}) \coloneqq \mathcal{H}_{\vec{\gamma},j}(\vec{f}) 
		&\lesssim  \sup_{h\in \mathfrak{b}_{X',u}} \int_{\rn} \sum_{Q \in \mathcal{S}_j} \prod_{i=1}^{l} \langle |(b_{i} - \langle b_{i} \rangle_{Q}) f_{i} | \rangle  
		\prod_{i=l+1}^{m} \langle |f_{i}| \rangle_{Q} \chi_{Q}(x) |h(x)| dx\\
		&= \sup_{h\in \mathfrak{b}_{X',u}} \sum_{Q \in \mathcal{S}_j} \prod_{i=1}^{l} \langle |(b_{i}-\langle b_{i}\rangle_{Q}) f_{i}| \rangle_{Q} 
		\prod_{i=l+1}^{m} \langle |f_{i}| \rangle_{Q} \langle |h|\rangle_{Q} |Q|.
	\end{aligned}
	\end{equation}
	By iteration argument and \cite[Lemma 5.1]{ler1}, it is known that there exists a sparse family $\mathcal{S}_{j}^{\ast} \subset \mathscr{D}$ such that for any $1\leq i\leq l$, 
	$$|b_{i}(x) - \langle b_{i}\rangle_{Q}| \leq 2^{n+2} \sum_{R \in \mathcal{S}_{j}^{\ast}, R \subseteq Q} \langle\left|b_{i}-{\langle b_{i} \rangle}_{R}\right|\rangle_{R} \chi_{R}(x).$$
	Hence, 
	$$\frac{1}{|Q|} \int_{Q} |b_{i}(x)-\langle b_{i}\rangle_{Q}| |f_{i}(x)| dx \leq C_{n} \|b_{i}\|_{\text{BMO}} \frac{1}{|Q|} \int_{Q} \mathcal{T}_{\mathcal{S}_{j}^{\ast}}(f_{i})(x) dx,$$
	where $\mathcal{T}_{\mathcal{S}_{j}^{\ast}}$ denotes the classical sparse operator. Utilizing \eqref{eq2-1} and Carleson embedding theorem, we have
	\begin{align*}
		\mathcal{H}_{\vec{2},j}(\vec{f}) 
		&\lesssim \sup_{h\in \mathfrak{b}_{X',u}} 
		\prod_{i=1}^{l} \|b_{i}\|_{\text{BMO}} 
		\sum_{Q \in \mathcal{S}_j} 
		\prod_{i=1}^{l} \langle\mathcal{T}_{\mathcal{S}_{j}^{\ast}} (f_{i})\rangle_{Q} 
		\prod_{i=l+1}^{m} \langle|f_{i}|\rangle_{Q} \langle|h|\rangle_{Q} |Q|\\
		&\lesssim \prod_{i=1}^{l} \|b_{i}\|_{\text{BMO}}
		\sup_{h\in \mathfrak{b}_{X',u}} \int_{\rn} \mathscr{M}(\vec{f}_{0})(x) Mh(x) dx,
	\end{align*}
	where $\vec{f}_{0} \coloneqq (\mathcal{T}_{\mathcal{S}_{j}^{\ast}}(f_{1}), \dots, \mathcal{T}_{\mathcal{S}_{j}^{\ast}}(f_{l}), f_{l+1}, \dots,f_{m})$. Then the H\"older inequality \eqref{holder-M-B}, combined with Theorem \ref{the boundedness of the Hardy-Littlewood maximal operator on {B}_{X',u}}, leads to
	\begin{equation}
		\begin{aligned}\label{eq2-2}
		\mathcal{H}_{\vec{2},j}(\vec{f}) 
		&\lesssim \prod_{i=1}^{l} \|b_{i}\|_{\text{BMO}} \|\mathscr{M}(\vec{f}_{0})\|_{M_{X}^{u}}
		\sup_{h\in \mathfrak{b}_{X',u}} \|Mh\|_{\mathfrak{B}_{X',u}}\\
		&\leq \prod_{i=1}^{l} \|b_{i}\|_{\text{BMO}} \|\mathscr{M}(\vec{f}_{0})\|_{M_{X}^{u}} 
		\|M\|_{\mathfrak{B}_{X',u}}
		\sup_{h\in \mathfrak{b}_{X',u}} \|h\|_{\mathfrak{B}_{X',u}}\\
		&= \|\vec{b}\|_{\text{BMO}} \|M\|_{\mathfrak{B}_{X',u}} \|\mathscr{M}(\vec{f}_{0})\|_{M_{X}^{u}}.
	\end{aligned}
	\end{equation}
Note that $\mathscr{M}$ is bound from $M_{X_{1}}^{u_{1}} \times \cdots \times M_{X_{m}}^{u_{m}}$ to $M_{X}^{u}$, therefore, 
	\begin{align}\label{eq2-3}
		\|\mathscr{M}(\vec{f}_{0})\|_{M_{X}^{u}} \leq \left\|\mathscr{M}\right\|_{M_{X_{1}}^{u_{1}} \times \cdots \times M_{X_{m}}^{u_{m}} \to M_{X}^{u}}
		\prod_{i=1}^{l} \left\|\mathcal{T}_{\mathcal{S}_{j}^{\ast}} (f_{i}) \right\|_{M_{X_{i}}^{u_{i}}} 
		\prod_{i=l+1}^{m} \|f_{i}\|_{M_{X_{i}}^{u_{i}}}.
	\end{align}
To proceed with the proof, we first establish the following claim, whose proof will be postponed until the end of the proof of Theorem \ref{thm2}.
	\begin{claim}\label{claim1}
		Let $\mathcal{S}$ be any sparse family, $u\in \mathbb{W}_{X}^{0}, X\in \mathbb{M}$ and $X'\in \mathbb{M}$. Then $\mathcal{T}_{\mathcal{S}}(f)$ satisfies
		$$\|\mathcal{T}_{\mathcal{S}}(f)\|_{M_{X}^{u}} \lesssim \|M\|_{M_{X}^{u}} \|M\|_{\mathfrak{B}_{X',u}} \|f\|_{M_{X,u}}.$$
	\end{claim}
In fact, combining \eqref{eq2-2}, \eqref{eq2-3}, and Claim \ref{claim1} yields
	$$\mathcal{H}_{\vec{2},j}(\vec{f}) \lesssim \|\vec{b}\|_{\text{BMO}} \|M\|_{\mathfrak{B}_{X',u}} \|\mathscr{M}\|_{M_{X_{1}}^{u_{1}} \times \cdots \times M_{X_{m}}^{u_{m}} \to M_{X}^{u}} 
	\left(\prod_{i=1}^{l} \|M\|_{\mathfrak{B}_{X_{i}',u_{i}}} \|M\|_{M_{X_{i}}^{u_{i}}}\right)
	\prod_{i=1}^{m} \|f_{i}\|_{M_{X_{i}}^{u_{i}}}.$$

    \noindent (2) For $l_{1}=l$, we have $\vec{\gamma}=(1,\dots,1)$. Thus, 
    \begin{equation}
	\begin{aligned}\label{eq2-4}
	    \mathcal{H}_{\vec{1},j}(\vec{f}) \coloneqq \mathcal{H}_{\vec{\gamma},j}(\vec{f}) &\lesssim \sup_{h\in \mathfrak{b}_{X',u}} \int_{\rn} \sum_{Q \in \mathcal{S}_j}
	    \prod_{i=1}^{l} |b_{i}(x)-\langle b_{i}\rangle_{Q}| \prod_{i=1}^{m} \langle|f_{i}|\rangle_{Q} \chi_{Q}(x) |h(x)| dx\\
	    &= \sup_{h\in \mathfrak{b}_{X',u}}
	    \sum_{Q \in \mathcal{S}_j} 
	    \Big\langle \prod_{i=1}^{l} |b_{i}(x)-\langle b_{i}\rangle_{Q} | |h(x)|\Big\rangle_{Q}
	    \prod_{i=1}^{m} \langle|f_{i}|\rangle_{Q} 
	    |Q|.
    \end{aligned}
    \end{equation}
    Invoking the generalized H\"older inequality (\!\!\cite[Lemma 4.1]{bp+tan}) yields that 
    $$\frac{1}{|Q|} \int_{Q} \prod_{i=1}^{l} |b_{i}(x) - \langle b_{i}\rangle_{Q}| |h(x)| dx \lesssim \prod_{i=1}^{l} \|b_{i}-\langle b_{i}\rangle_{Q}\|_{\exp{L},Q} \|h\|_{L(\log L)^{l},Q}.$$
   Noting that \( \|b_{i}-\langle b_{i}\rangle_{Q}\|_{\exp{L},Q} \lesssim \|b_{i}\|_{\text{BMO}} \) for \( i=1,\dots,l \), we can deduce from \eqref{eq2-4} that
   
    \begin{align*}
	    \mathcal{H}_{\vec{1},j}(\vec{f}) 
	    &\lesssim \prod_{i=1}^{l} \|b_{i}\|_{\text{BMO}} \sup_{h\in \mathfrak{b}_{X',u}} \sum_{Q \in \mathcal{S}_j} \prod_{i=1}^{m} \langle |f_{i}|\rangle_{Q} \|h\|_{L(\log L)^{l},Q} |Q|\\
	    &\lesssim \|\vec{b}\|_{\text{BMO}}
	    \sup_{h\in \mathfrak{b}_{X',u}} \sum_{Q \in \mathcal{S}_j} \inf_{x\in Q} \mathscr{M}(\vec{f})(x) 
	    \inf_{x\in Q} M_{L(\log L)^{l}} h(x) |E_{Q}|\\
	    &\lesssim \|\vec{b}\|_{\text{BMO}} \sup_{h\in \mathfrak{b}_{X',u}} \int_{\rn} \mathscr{M}(\vec{f})(x) M_{L(\log L)^{l}} h(x) dx,
    \end{align*}
    where the sets $\{E_{Q}\}$ are pairwise disjoint and $|E_{Q}| \geq \eta |Q|$ with $\eta$ a constant only depends on $n$. 
    Let $M^{l+1}$ denote the $(l+1)$-th composition of the Hardy-Littlewood maximal operator $M$, then $M_{L(\log L)^{l}} g(x) \simeq M^{l+1} g(x)$. This fact yields that 
    \begin{align*}
	    \mathcal{H}_{\vec{1},j}(\vec{f}) &\lesssim \|\vec{b}\|_{\text{BMO}} \|\mathscr{M}(\vec{f})\|_{M_{X}^{u}} 
	    \sup_{h\in \mathfrak{b}_{X',u}} \|M^{l+1}h\|_{\mathfrak{B}_{X',u}}\\
	    & \lesssim \|\vec{b}\|_{\text{BMO}} \|M\|_{\mathfrak{B}_{X',u}}^{l+1} \|\mathscr{M}\|_{M_{X_{1}}^{u_{1}} \times \cdots \times M_{X_{m}}^{u_{m}} \to M_{X}^{u}} \prod_{i=1}^{m} \|f_{i}\|_{M_{X_{i}}^{u_{i}}}.
    \end{align*}
    
    \noindent (3) For $0<l_{1}<l$, it is obvious that 
    \begin{align*}
	    \mathcal{H}_{\overrightarrow{l_{1}},j}(\vec{f}) \coloneqq \mathcal{H}_{\vec{\gamma},j}(\vec{f}) 
	    &\lesssim \sup_{h\in \mathfrak{b}_{X',u}} \int_{\rn} \sum_{Q \in \mathcal{S}_j} \prod_{i=1}^{l_{1}} \langle|f_{i}|\rangle_{Q}
	    \prod_{i=l_{1}+1}^{l} \langle |(b_{i}-\langle b_{i}\rangle_{Q}) f_{i}| \rangle_{Q}\\
	    &\qquad \qquad \times \prod_{i=l+1}^{m} \langle |f_{i}|\rangle_{Q}  \prod_{i=1}^{l_{1}} |b_{i}(x) -\langle b_{i} \rangle_{Q}| \chi_{Q}(x) |h(x)| dx\\
	    &=\sup_{h\in \mathfrak{b}_{X',u}} \sum_{Q \in \mathcal{S}_j} \int_{Q} \prod_{i=1}^{l_{1}} |b_{i}(x)-\langle b_{i}\rangle_{Q}| |h(x)| dx 
	    \prod_{i\notin \{l_{1}+1, \dots,l\}} \langle|f_{i}|\rangle_{Q}\\
	    &\qquad \qquad \times \prod_{i=l_{1}+1}^{l} \langle |(b_{i}-\langle b_{i}\rangle_{Q}) f_{i}|\rangle_{Q}.
    \end{align*}
  
  In view of the inequality  $\frac{1}{|Q|} \int_{Q} |b_{i}(x) -\langle b_{i}\rangle_{Q}| |f_{i}(x)| dx \lesssim \|b_{i}-\langle b_{i}\rangle_{Q}\|_{\exp{L},Q} \|f_{i}\|_{L(\log L),Q}$, we obtain 
    \begin{align*}
	    \mathcal{H}_{\overrightarrow{l_{1}},j}(\vec{f}) &\lesssim \|\vec{b}\|_{\text{BMO}} 
	    \sup_{h\in \mathfrak{b}_{X',u}} \sum_{Q \in \mathcal{S}_j} \prod_{i\notin \{l_{1}+1, \dots,m\}} \langle|f_{i}|\rangle_{Q} 
	    \prod_{i=l_{1}+1}^{l} \|f_{i}\|_{L(\log L),Q} \|h\|_{L(\log L)^{l_{1}},Q} |Q|\\
	    &\lesssim \|\vec{b}\|_{\text{BMO}} 
	    \sup_{h\in \mathfrak{b}_{X',u}} \int_{\rn} \mathscr{M}_{L(\log L)}^{(l)} (\vec{f})(x) M_{L(\log L)^{l_{1}}} h(x) dx,
    \end{align*}
    where $\mathscr{M}_{L(\log L)}^{(l)} (\vec{f})(x) \coloneqq \sup_{Q \owns x} \prod_{i=1}^{l} \|f_{i}\|_{L(\log L),Q} \prod_{i=l+1}^{m} \langle |f_{i}|\rangle_{Q}$. 
    The reuse of the H\"older's inequality gives
    $$ \mathcal{H}_{\overrightarrow{l_{1}},j}(\vec{f}) \lesssim \|\vec{b}\|_{\text{BMO}} \big\|\mathscr{M}_{L(\log L)}^{(l)} (\vec{f})\big\|_{M_{X}^{u}} 
    \sup_{h\in \mathfrak{b}_{X',u}} \big\| M_{L(\log L)^{l_{1}}} h\big\|_{\mathfrak{B}_{X',u}}.$$
  Noting that  $l_{1}<l$ and 
    $$\mathscr{M}_{L(\log L)}^{(l)} (\vec{f})(x) \leq C_{n} \mathscr{M}\left(Mf_{1},\dots,Mf_{l},f_{l+1},\dots,f_{m}\right)(x),$$
    one can arrive at 
    $$\mathcal{H}_{\overrightarrow{l_{1}},j}(\vec{f}) \lesssim \|\vec{b}\|_{\text{BMO}} \|\mathscr{M}\|_{{M_{X_{1}}^{u_{1}} \times \cdots \times M_{X_{m}}^{u_{m}} \to M_{X}^{u}}} \left(\prod_{i=1}^{l} \|M\|_{M_{X_{i}}^{u_{i}}}\right) 
    \| M\|_{\mathfrak{B}_{X',u}}^{l+1} 
    \prod_{i=1}^{m} \|f_{i}\|_{M_{X_{i}}^{u_{i}}}.$$
    
    Taking all cases and \eqref{eq2-6} into account, we conclude that
    \begin{align*}
	    \left\| \mathcal{T}_{\vec{b}} (\vec{f}) \right\|_{M_{X}^{u}} &\lesssim \sum_{j=1}^{3^{n}} \left(\mathcal{H}_{\vec{1},j}(\vec{f}) + \mathcal{H}_{\vec{2},j}(\vec{f}) + \max_{0<l_{1}<l} \left\{\mathcal{H}_{\overrightarrow{l_{1}},j}(\vec{f})\right\}\right)\\
	    &\lesssim C_{M} \|\vec{b}\|_{\text{BMO}} \|\mathscr{M}\|_{M_{X_{1}}^{u_{1}} \times \cdots \times M_{X_{m}}^{u_{m}} \to M_{X}^{u}} 
	    \prod_{i=1}^{m} \|f_{i}\|_{M_{X_{i}}^{u_{i}}},
    \end{align*}
    where $C_{M} = \|M\|_{\mathfrak{B}_{X',u}} \left(\prod_{i=1}^{l} \left(\|M\|_{\mathfrak{B}_{X_{i}',u_{i}}} \|M\|_{M_{X_{i}}^{u_{i}}}\right) + \|M\|_{\mathfrak{B}_{X',u}}^{l+1} +\|M\|_{\mathfrak{B}_{X',u}}^{l} \prod_{i=1}^{l} \|M\|_{M_{X_{i}}^{u_{i}}}\right)$.
    
   All that remains in the proof is to verify Claim \ref{claim1}. In fact, \eqref{mxu<sup<mxu} implies that
    \begin{align*}
	    \left\|\mathcal{T}_{\mathcal{S}} (f)\right\|_{M_{X}^{u}} 
	    &\leq \frac{1}{C_{0}} \sup_{h\in \mathfrak{b}_{X',u}} \int_{\rn} \left|\mathcal{T}_{\mathcal{S}} f(x)\right| |h(x)| dx\\
	    &= \frac{1}{C_{0}} \sup_{h\in \mathfrak{b}_{X',u}} \int_{\rn} \sum_{Q \in \mathcal{S}} \langle |f|\rangle_{Q} \chi_{Q}(x) |h(x)| dx\\
	    &=\frac{1}{C_{0}} \sup_{h\in \mathfrak{b}_{X',u}} \sum_{Q \in \mathcal{S}} \langle|f|\rangle_{Q} \langle|h|\rangle_{Q} |Q|.
    \end{align*}
   The sparse property of $\mathcal{S}$ allows us to deduce that
    \begin{align*}
	    \left\|\mathcal{T}_{\mathcal{S}}(f)\right\|_{M_{X}^{u}} &\lesssim \sup_{h\in \mathfrak{b}_{X',u}} \sum_{Q \in \mathcal{S}} \inf_{x\in E_{Q}} Mf(x) \inf_{x\in E_{Q}} Mh(x) |E_{Q}|\\
	    &\leq \sup_{h\in \mathfrak{b}_{X',u}} \sum_{Q \in \mathcal{S}} \int_{E_{Q}} Mf(x) Mh(x) dx\\
	    &\leq \sup_{h\in \mathfrak{b}_{X',u}} \int_{\rn} Mf(x) Mh(x) dx.
    \end{align*}
   Noting that  $\|h\|_{\mathfrak{B}_{X',u}}\leq 1$, it follows from\eqref{holder-M-B} that
    \begin{align*}
	    \left\|\mathcal{T}_{\mathcal{S}}(f)\right\|_{M_{X}^{u}} &\lesssim \|Mf\|_{M_{X}^{u}} \sup_{h\in \mathfrak{b}_{X',u}} \|Mh\|_{\mathfrak{B}_{X',u}}\\
	    &\leq \|M\|_{M_{X}^{u}} \|M\|_{\mathfrak{B}_{X',u}} \|f\|_{M_{X}^{u}},
    \end{align*}
   which proves Claim \ref{claim1}. The proof of Theorem \ref{thm2} is complete.
\end{proof}

\begin{remark}
It is noteworthy that the exponent \( l+1 \) in the constant \( \|M\|_{\mathfrak{B}_{X',u}}^{l+1} \), obtained through the above method, is smaller than the exponent \( l+2 \) derived via extrapolation. This improvement arises because sparse domination circumvents the need for the weighted Coifman-Fefferman inequality (Lemma \ref{yinli3}), thereby yielding a better constant.

\end{remark}

\section{Proof of Theorems \ref{thm-rough}}\label{pr+last}
In this section, we establish the boundedness of rough singular integral operators on Morrey-Banach spaces. This requires the following two lemmas, excerpted from \cite[Theorem 3.3]{ler8} and \cite[Lemma 3.4]{ler8}.
\begin{lemma}[\!\!\cite{ler8}]\label{lem-new1}
	Let $1\leq r<s\leq \infty$ and $m\in \mathbb{N}^{+}$. Suppose $T$ is a sublinear operator and $T, M_{T,s}^{\sharp}$ enjoy the $W_{r}$ property. Then there exist $C_{m,n}>1$ and $\lambda_{m,n}<1$ such that for any bounded functions $f,g$ with compact supports and $b\in L_{\text{loc}}^{1}(\rn)$, there exists a $\frac{1}{2\cdot 3^{n}}$-sparse family $\mathcal{S}$ satisfying 
	$$\int_{\rn} \left|T_{b}^{m} f(x) g(x)\right| dx \leq C\sum_{k=0}^{m} \Big(\sum_{Q \in \mathcal{S}} \langle |b-\langle b\rangle_{Q}|^{m-k} |f| \rangle_{r,Q}  \langle |b-\langle b\rangle_{Q}|^{k} |g|\rangle_{s',Q} |Q| \Big),$$
	where $C\coloneqq C_{m,n}(\varphi_{T,r}(\lambda_{m,n}) + \varphi_{M_{T,s}^{\sharp}, r}(\lambda_{m,n}))$.
\end{lemma}
\begin{remark}
	Taking $m=1$ in Lemma \ref{lem-new1} and Hypothesis \ref{hyp}, it is easy to check that Hypothesis \ref{hyp} fulfills Lemma \ref{lem-new1}.
\end{remark}
\begin{lemma}[\!\!\cite{ler8}]\label{lem-new2}
	Let $1\leq r,t<\infty$ and $m\in \mathbb{N}^{+}$. Suppose bounded functions $f,g$ have compact supports and $b\in L_{\text{loc}}^{1}(\rn)$. Given any cube $Q$ and $0\leq k\leq m$, define 
	$$c_{k} \coloneqq \langle |b-\langle b\rangle_{Q}|^{m-k} |f|\rangle_{r,Q}  \langle|b-\langle b\rangle_{Q}|^{k} |g|\rangle_{t,Q},$$
	then $c_{k} \leq c_{0}+c_{m}$.
\end{lemma}
Now, we can proceed to prove that rough singular integral operators are bounded on Morrey-Banach spaces. 
\begin{proof}[Proof of Theorem \ref{thm-rough}]
	As can be seen from \eqref{mxu<sup<mxu}, we have
	$$\left\|T_{b}^{m} f\right\|_{M_{X}^{u}} \leq \frac{1}{c_{0}} \sup_{\|h\|_{\mathfrak{B}_{X',u}}\leq 1} \int_{\rn} |T_{b}^{m} f(x)| |h(x)| dx.$$
	Fixed $h\in \mathfrak{B}_{X',u}$ with $\|h\|_{\mathfrak{B}_{X',u}} \leq 1$, the joint application of Lemmata \ref{lem-new1} and \ref{lem-new2} yields that
	\begin{equation}
		\begin{aligned}\label{eq-new2}
		\int_{\rn} |T_{b}^{m} f(x)| |h(x)| dx 
		&\lesssim \sum_{Q \in \mathcal{S}} \langle|b-\langle b\rangle_{Q}|^{m} |f|\rangle_{r,Q} \langle|h|\rangle_{s',Q} |Q| \\
		    &\quad + \sum_{Q \in \mathcal{S}} \langle|f|\rangle_{r,Q} \langle|b-\langle b\rangle_{Q}|^{m} |h|\rangle_{s',Q} |Q|\\
		&\eqqcolon I_{1}^{0} +I_{2}^{0},
	\end{aligned}
	\end{equation}
	where $I_{1}^{0} = \sum_{Q \in \mathcal{S}} \langle|b-\langle b\rangle_{Q}|^{m} |f|\rangle_{r,Q} \langle|h|\rangle_{s',Q} |Q|$ and $I_{2}^{0} = \sum_{Q \in \mathcal{S}} \langle|f|\rangle_{r,Q} \langle|b-\langle b\rangle_{Q}|^{m} |h|\rangle_{s',Q} |Q|$. 
	
	Since \( r_0 > \max\{r, s'\} \), an application of Hölder’s inequality yields
	
	\begin{align*}
		I_{1}^{0} &\leq \sum_{Q \in \mathcal{S}} \langle|b-\langle b\rangle_{Q}|^{m}\rangle_{r_{0}',Q}  \langle |f|\rangle_{r_{0},Q}  \langle|h|\rangle_{r,Q} |Q|\\
		&\lesssim \|b\|_{\text{BMO}}^{m}  \sum_{Q \in \mathcal{S}} \langle|f|\rangle_{r_{0},Q} \langle |h|\rangle_{r_{0},Q} |E_{Q}|,
	\end{align*}
	where the fact that \( \langle |b - \langle b \rangle_{Q}| \rangle_{p, Q} \simeq \|b\|_{\text{BMO}} \) for \( 0 < p < \infty \) is used.
	
	Notice that the sets $\{E_{Q}\}$ are pairwise disjoint. It follows that 
	\begin{align}\label{eq-new1}
		I_{1}^{0} \lesssim \|b\|_{\text{BMO}}^{m} \int_{\rn} M_{r_{0}} f(x) M_{r_{0}} h(x) dx
		\leq \|b\|_{\text{BMO}}^{m}  \|M_{r_{0}} f\|_{M_{X}^{u}}  \|M_{r_{0}} h\|_{\mathfrak{B}_{X',u}}.
	\end{align}
	Invoking that $u\in \mathbb{W}_{X',r_{0}^{-1}}$ and \cite[Theorem 2.11]{ho8} arrives at the boundedness of $M_{r_{0}}$ on $\mathfrak{B}_{X',u}$, moreover, 
	\begin{align*}
		\|M_{r_{0}} f\|_{M_{X}^{u}} &= \big\| \left(M(|f|^{r_{0}})\right)^{\frac{1}{r_{0}}} \big\|_{M_{X}^{u}}\\
		&\leq \|M\|_{M_{X^{1/r_{0}}}^{u^{r_{0}}}}^{\frac{1}{r_{0}}}  \| |f|^{r_{0}} \|_{M_{X^{1/r_{0}}}^{u^{r_{0}}}}^{\frac{1}{r_{0}}}\\
		&= \|M\|_{M_{X^{1/r_{0}}}^{u^{r_{0}}}}^{\frac{1}{r_{0}}}  \|f\|_{M_{X}^{u}},
	\end{align*}
	which combined with \eqref{eq-new1} indicates 
	$$I_{1}^{0} \lesssim \|b\|_{\text{BMO}}  \|M\|_{M_{X^{1/r_{0}}}^{u^{r_{0}}}}^{\frac{1}{r_{0}}} \|M_{r_{0}}\|_{\mathfrak{B}_{X',u}} \|f\|_{M_{X}^{u}}.$$
	
	Regarding $I_{2}^{0}$,  in a manner analogous to the treatment of $I_{1}^{0}$, we obtain 
	\begin{equation}
		\begin{aligned}\label{eq-new3}
		I_{2}^{0} &\lesssim \sum_{Q \in \mathcal{S}} \langle|b-\langle b\rangle_{Q}|^{m} \rangle_{r_{0}',Q}  \langle|f|\rangle_{r,Q}  \langle|h|\rangle_{r_{0},Q} |Q|\\
		&\lesssim \|b\|_{\text{BMO}}^{m}  \int_{\rn} M_{r_{0}} f(x)  M_{r_{0}} h(x) dx\\
		&\lesssim \|b\|_{\text{BMO}}^{m}  \|M\|_{M_{X^{1/r_{0}}}^{u^{r_{0}}}}^{\frac{1}{r_{0}}} \|M_{r_{0}}\|_{\mathfrak{B}_{X',u}} \|f\|_{M_{X}^{u}}. 
	\end{aligned}
	\end{equation}
	
	By integrating \eqref{eq-new3} with \eqref{eq-new1}, we have 
	$$\|T_{b}^{m}\|_{M_{X}^{u}} \lesssim \|b\|_{\text{BMO}}^{m}  \|M\|_{M_{X^{1/r_{0}}}^{u^{r_{0}}}}^{\frac{1}{r_{0}}} \|M_{r_{0}}\|_{\mathfrak{B}_{X',u}} \|f\|_{M_{X}^{u}},$$
which is precisely what we set out to prove.
\end{proof}

\section{Applications for $u$-Morrey-Lorentz spaces}\label{pr+3}
In the 1980s, Adams \cite{ada2} generalized Lorentz space to Morrey-Lorentz space as following form
$$R^{(p,q,\lambda)}\coloneqq \left\{f\in \mathcal{M}(\mathbb{R}^n): \|f\|_{R^{(p,q,\lambda)}} = \sup_{x\in \mathbb{R}^n, r>0} \frac{1}{r^{\lambda/p}} \|f\chi_{B(x,r)}\|_{L^{p,q}(\mathbb{R}^n)}<\infty\right\}$$
where $1\leq p<\infty, 0<q<\infty$ and $0\leq \lambda \leq n$.
Later on, Ragusa \cite{rag} extended the above definition of $\mathcal{R}^{(p,q,\lambda)}$ to the case where $0<p,q\leq\infty$ and defined the Morrey-Lorentz space $\mathcal{R}^{(p,q,\lambda)}$ by
$$\mathcal{R}^{(p,q,\lambda)}\coloneqq \left\{f\in \mathcal{M}(\mathbb{R}^n): \|f\|_{\mathcal{R}^{(p,q,\lambda)}} = \sup_{x\in \mathbb{R}^n, r>0} \frac{1}{r^{\lambda/q}} \|f\chi_{B(x,r)}\|_{L^{p,q}(\mathbb{R}^n)}<\infty\right\}$$
with $0\leq \lambda<n$. Notice that $R^{(p,q,\lambda \frac{p}{q})} = \mathcal{R}^{(p,q,\lambda)}$ and the Morrey-Lorentz space we mentioned later refers to $\mathcal{R}^{(p,q,\lambda)}$. In a certain sense, $\mathcal{R}^{(p,1,\lambda)}$ is the ``smallest" normed linear space and $\mathcal{R}^{(p,\infty,\lambda)}$ is the ``largest" one, that is, for any $1<q<\infty$, 
$$\mathcal{R}^{(p,1,\lambda)} \subseteq \mathcal{R}^{(p,q,\lambda)} \subseteq \mathcal{R}^{(p,\infty,\lambda)}.$$ 
Ragusa also pointed out that for $0\leq \lambda<n, 1<p<\infty$ and $1\leq q\leq \infty$, $\|\cdot\|_{\mathcal{R}^{(p,q,\lambda)}}$ is a norm under which $\mathcal{R}^{(p,q,\lambda)}$ is complete and similarly for $p=q=1, 0\leq \lambda<n$ or $p=q=\infty, 0\leq \lambda<n$. 
In 2013, Aykol, Guliyev, and Serbetci \cite{ayk} introduced the concept of the local Morrey-Lorentz space and considered the boundedness of the maximal operator on this space. More recently, Ho \cite{ho3} investigated the boundedness of nonlinear commutators on Morrey-Lorentz spaces.

Note that $r^{\frac{\lambda}{q}}$ plays a crucial role in the definition of $\mathcal{R}^{(p,q,\lambda)}$. Therefore, we can generalize it to the $u$-Morrey-Lorentz space naturally.
\begin{definition}[$u$-Morrey-Lorentz space]\label{def-uml}
	Given $f\in \mathcal{M}(\mathbb{R}^n)$ and $0<p,q\leq \infty$, define
	$$\|f\|_{M_{p,q}^u} \coloneqq  \sup_{x\in \mathbb{R}^n, r>0} \frac{1}{u(x,r)} \left\|f\chi_{B(x,r)}\right\|_{L^{p,q}(\mathbb{R}^n)}.$$
	The set of all $f$ with $\|f\|_{M_{p,q}^u}<\infty$ is denoted by $M_{p,q}^u$ and is called the $u$-Morrey-Lorentz space. 
\end{definition}

\begin{remark}
	If $u(x,r) = r^{\frac{\lambda}{q}}$ with $0\leq \lambda<n$, then the space $M_{p,q}^u$ returns to the Morrey-Lorentz space. If $1\leq p=q<\infty$, then the space $M_{p,q}^u$ is the classical Morrey space. 
\end{remark}

In this section, the main result is the boundedness of the iterated commutator of the multilinear operator on the $u$-Morrey-Lorentz space.
\begin{theorem}\label{thm3}
	Let $I \coloneqq  \{i_1,\dots,i_l\} = \{1,\dots,l\} \subseteq \{1,\dots,m\}$, $1< p,p_1,\dots,p_m<\infty, 1\leq q,q_1,\dots,q_m<\infty$ satisfying $\frac{1}{p} = \sum_{i=1}^{m} \frac{1}{p_i}$ and $\frac{1}{q} = \sum_{i=1}^{m} \frac{1}{q_i}$, and $0<\theta<\min_{i\geq 1}\{\frac{1}{p_i}\}$. Suppose that $\vec{b} \in \mathrm{BMO}^l$ and $\mathcal{T}_{\vec{b}}$ satisfies Hypothesis \ref{hyp}. If there exists a constant $C>0$ such that for any $x,x_1,x_2\in \mathbb{R}^n, r,r_1,r_2\in \mathbb{R}^+, i\in \{0,\dots,m\}$,
	\begin{align}\label{eq4}
		\frac{\|\chi_{B(x_1,r_1)}\|_{L^{p_i,q_i}(\mathbb{R}^n)}}{u_i(x_1,r_1)} \leq C \frac{\|\chi_{B(x_2,r_2)}\|_{L^{p_i,q_i}(\mathbb{R}^n)}}{u_i(x_2,r_2)} \ \text{when} \ u_i(x_1,r_1) \leq u_i(x_2,r_2),
	\end{align}
	and
	\begin{align}\label{eq5}
		u_i(x,2^jr) \leq C 2^{jn\theta} u_i(x,r),
	\end{align}
	where $p_0\coloneqq p, q_0\coloneqq q, u_0(x,r)\coloneqq u(x,r)=\prod_{i=1}^m u_i(x,r)$, then
	\begin{align*}
		\left\|\mathcal{T}_{\vec{b}}(\vec{f})\right\|_{M_{p,q}^{u}} 
		\leq C \|\vec{b}\|_{\text{BMO}} \|M\|_{\mathfrak{B}_{L^{p',q'},u}}^{2+l} \|M\|_{M_{p_1,q_1}^{u_1} \times \cdots \times M_{p_m,q_m}^{u_m} \to M_{p,q}^{u}} 
		\prod_{i=1}^m \|M\|_{M_{p_i,q_i}^{u_i} \to M_{p_i,q_i}^{u_i}}  \|f_i\|_{M_{p_i,q_i}^{u_i}}.
	\end{align*}
\end{theorem}
\begin{proof}[Proof of Theorem \ref{thm3}]
	For any $1<p,p_1,\dots,p_m<\infty, 1\leq q,q_1,\dots,q_m<\infty$ and any ball $B(x,r)$, it is easy to see 
	\begin{align}\label{chib-lpq}
		\|\chi_{B(x,r)}\|_{L^{p,q}(\mathbb{R}^n)} = \left(\frac{p}{q}\right)^{\frac{1}{q}} v_n^{\frac{1}{p}} r^{\frac{n}{p}} = \left(\frac{p}{q}\right)^{\frac{1}{q}} \prod_{i=1}^m v_n^{\frac{1}{p_i}} r^{\frac{n}{p_i}},
	\end{align}
	where $v_n$ denotes the volume of the unit ball in $\rn$. According to \cite[Propsition 1.4.16]{gra}, the associate space of $L^{p,q}(\mathbb{R}^n)$ is $L^{p',q'}(\mathbb{R}^n)$ for $1<p,q<\infty$, and the associate space of $L^{p,1}(\mathbb{R}^n)$ is $L^{p',\infty}(\mathbb{R}^n)$ for any $1<p<\infty$. By taking $w\equiv1$ in Lemma \ref{acc-cor2.2}, we can see that the Hardy-Littlewood maximal operator $M$ is bounded on $L^{p',q'}(\mathbb{R}^n)$ for any $1<p<\infty$ and $1\leq q< \infty$, which implies that $L^{p,q}(\mathbb{R}^n)\in \mathbb{M}'$ and $L^{p_i,q_i}(\mathbb{R}^n) \in \mathbb{M}'$ for $i=1,\dots,m$. Combining the fact $\mathscr{M}(\vec{f})(x)\leq \prod_{i=1}^m Mf_i(x)$ with Lemma \ref{lem4}, we obtain
	\begin{align*}
		\|\mathscr{M}(\vec{f})\|_{L^{p,q}(\mathbb{R}^n)} \leq m^{\frac{1}{p}} \prod_{i=1}^m \|Mf_i\|_{L^{p_i,q_i}(\mathbb{R}^n)}\lesssim C_{q,n} \prod_{i=1}^m (p_i')^{\frac{1}{\min\{p_i,q_i\}}} m^{\frac{1}{p_i}} \|f_i\|_{L^{p_i,q_i}(\mathbb{R}^n)},
	\end{align*}
	where $1<p,p_1,\dots,p_m<\infty$ and $1\leq q,q_1,\dots,q_m<\infty$.
	
	In order to apply Theorem \ref{thm2}, it suffices to verify that $u_0\coloneqq u\in \mathbb{W}_{L^{p,q}(\mathbb{R}^n)}^0$ and $u_i\in \mathbb{W}_{L^{p_i,q_i}(\mathbb{R}^n)}^0, i=1,\dots,m$. Notice that $p_0=p, q_0=q$ and $0<\theta<\min_{i\geq 1}\left\{\frac{1}{p_i}\right\}<\frac{1}{p}$ since that $\frac{1}{p} = \sum_{i=1}^m \frac{1}{p_i}$. For any $u_i, i=0,\dots,m$, combined with \eqref{eq4}, what we need to do is to estimate the term (2) in Definition \ref{def-wxalpha}. Indeed, \eqref{eq5} and \eqref{chib-lpq} guarantee that 
	\begin{equation*}
		\begin{aligned}
		\sum_{j=0}^{\infty} \frac{\|\chi_{B(x,r)}\|_{L^{p_i,q_i}(\mathbb{R}^n)}}{\|\chi_{B(x,2^{j+1}r)}\|_{L^{p_i,q_i}(\mathbb{R}^n)}} \frac{u_{i}(x,2^{j+1}r)}{u_i(x,r)} 
		&\leq C \sum_{j=0}^{\infty} 2^{(j+1)n\theta} \frac{\|\chi_{B(x,r)}\|_{L^{p_i,q_i}(\mathbb{R}^n)}}{\|\chi_{B(x,2^{j+1}r)}\|_{L^{p_i,q_i}(\mathbb{R}^n)}} \\
		&= C \sum_{j=0}^{\infty} 2^{(j+1)n\left(\theta-\frac{n}{p}\right)}
		<\infty.
	\end{aligned}
	\end{equation*}
	Therefore, $u_i\in \mathbb{W}_{L^{p_i,q_i}(\mathbb{R}^n)}^0, i=1,\dots,m$ and $u=\prod_{i=1}^m u_i \in \mathbb{W}_{L^{p,q}(\mathbb{R}^n)}^0$. Following \cite[Page 5555]{bar}, $L^{p,q}(\mathbb{R}^n)$ is a normable Banach function space. Therefore, Theorem \ref{thm2} provides 
	\begin{align*}
		\left\|\mathcal{T}_{\vec{b}}(\vec{f})\right\|_{M_{p,q}^u} \leq C\|\vec{b}\|_{\text{BMO}} \|M\|_{\mathfrak{B}_{L^{p',q'},u}}^{2+l} \|M\|_{M_{p_1,q_1}^{u_1} \times \cdots \times M_{p_m,q_m}^{u_m} \to M_{p,q}^u} \prod_{i=1}^m \|M\|_{M_{p_i,q_i}^{u_i} \to M_{p_i,q_i}^{u_i}} \|f_i\|_{M_{p_i,q_i}^{u_i}},
	\end{align*}
	which completes the proof of Theorem \ref{thm3}.
\end{proof}

\begin{corollary}\label{cor1}
	Let $I \coloneqq  \{i_1,\dots,i_l\} = \{1,\dots,l\} \subseteq \{1,\dots,m\}$, $1< p,p_1,\dots,p_m<\infty, 1\leq q,q_1,\dots,q_m<\infty$ satisfying $\frac{1}{p} = \sum_{i=1}^{m} \frac{1}{p_i}, \frac{1}{q} = \sum_{i=1}^{m} \frac{1}{q_i}$, and $0\leq\lambda<n$ with $\frac{\lambda}{n} < \min\left\{\frac{q_i}{p_i}, \frac{q}{p}\right\}$. Suppose that $\vec{b} \in \mathrm{BMO}^l$ and $\mathcal{T}_{\vec{b}}$ satisfies Hypothesis \ref{hyp}. Then 
	\begin{align}\label{eq6}
		\left\|\mathcal{T}_{\vec{b}}(\vec{f})\right\|_{\mathcal{R}^{(p,q,\lambda)}}
	    \leq C \|\vec{b}\|_{\text{BMO}} 
	    \prod_{i=1}^m \|f_i\|_{\mathcal{R}^{(p_i,q_i,\lambda)}}.
	\end{align}
	In particular,
	\begin{align}\label{eq7}
		\left\|\mathcal{T}_{\vec{b}}(\vec{f})\right\|_{L^{q,\lambda}(\mathbb{R}^n)}
		\leq C \|\vec{b}\|_{\text{BMO}} 
		\prod_{i=1}^m \|f_i\|_{L^{q_i,\lambda}(\mathbb{R}^n)}.
	\end{align}
\end{corollary}
\begin{remark}
	There are three points worth mentioning. The first is that for $1 < p, p_1, \dots, p_m < \infty$ and $1 \leq q, q_1, \dots, q_m < \infty$, as a special case of Corollary \ref{cor1}, the operator $\mathcal{T}_{\vec{b}}$ is bounded from $L^{p_1, q_1}(\mathbb{R}^n) \times \cdots \times L^{p_m, q_m}(\mathbb{R}^n)$ to $L^{p, q}(\mathbb{R}^n)$, since $\mathcal{R}^{(p, q, 0)} = L^{p, q}(\mathbb{R}^n)$.
	 The second point is that Corollary \ref{cor1} generalizes Theorem 5.2 in \cite{zha} by removing the condition $q_i<p_i$ for $i=1, \dots, m$. The third point is that the condition $\frac{\lambda}{n}<\min_i \{ \frac{q_i}{p_i}, \frac{q}{p} \}$ is necessary, because the Morrey space  $L^{p, \lambda}(\mathbb{R}^n)$ requires $\lambda<n$ when $m=1$ and $p=q$.
\end{remark}

\begin{proof}[Proof of Corollary \ref{cor1}]
	For any $x\in \rn$ and $r>0$, let $u(x,r)=r^{\frac{\lambda}{q}}$ and $u_i(x,r)=r^{\frac{\lambda}{q_i}}$, then $M_{p,q}^u = \mathcal{R}^{(p,q,\lambda)}, M_{p_i,q_i}^\lambda = \mathcal{R}^{(p_i,q_i,\lambda)}$ with $0\leq \lambda<n$. In accordance with Theorems \ref{thm2} and \ref{thm3}, it suffices to verify $u\in \mathbb{W}_{L^{p,q}(\mathbb{R}^n)}^0$ and $u_i\in \mathbb{W}_{L^{p_i,q_i}(\mathbb{R}^n)}^0$.
	On the one hand, suppose that $u(x_1,r_1) \leq u(x_2,r_2)$, then $r_1\leq r_2$. Thus, \eqref{chib-lpq} asserts that
	\begin{align}\label{eq-1127-1}
		\frac{\|\chi_{B(x_1,r_1)}\|_{L^{p,q}(\mathbb{R}^n)}}{u(x_1,r_1)} 
		= Cr_1^{\frac{n}{p}-\frac{\lambda}{q}}
		\leq C r_2^{\frac{n}{p}-\frac{\lambda}{q}} 
		= C \frac{\|\chi_{B(x_2,r_2)}\|_{L^{p,q}(\mathbb{R}^n)}}{u(x_2,r_2)},
	\end{align} 
	which is obtained by using the fact that $\frac{\lambda}{n}<\frac{q}{p}$. On the other hand, we have
	\begin{align}\label{eq-1127-2}
		\sum_{j=0}^{\infty} \frac{\|\chi_{B(x,r)}\|_{L^{p,q}(\mathbb{R}^n)}}{\|\chi_{B(x,2^{j+1}r)}\|_{L^{p,q}(\mathbb{R}^n)}} \frac{u(x,2^{j+1}r)}{u(x,r)} 
		\leq C \sum_{j=0}^{\infty} 2^{(j+1)\left(\frac{\lambda}{q}-\frac{n}{p}\right)}
		<\infty.
	\end{align}
	Therefore, $u\in \mathbb{W}_{L^{p,q}(\mathbb{R}^n)}^0$. For $u_i$, $i = 1, \dots, m$, by combining $\frac{1}{q} = \sum{i=1}^m \frac{1}{q_i}$ with the condition $\frac{\lambda}{n} < \frac{q_i}{p_i}$, and using \eqref{eq-1127-1} and \eqref{eq-1127-2} for $u$, we deduce that $u_i \in \mathbb{W}_{L^{p_i,q_i}(\mathbb{R}^n)}^0$ and \eqref{eq6} follows. Taking $1 \leq p_i = q_i < \infty$ for $i = 1, \dots, m$ in Theorem \ref{thm3}, it is straightforward to verify that \eqref{eq7} holds.
\end{proof}

\section{Applications to Morrey Spaces with Variable Exponents}\label{pr+4}
Let $\mathcal{P}(\mathbb{R}^n)$ denote the set of all Lebesgue measurable functions $p(\cdot): \mathbb{R}^n \to [1, +\infty]$. The elements of $\mathcal{P}(\mathbb{R}^n)$ are called exponent functions (or simply exponents). Some notation for describing the range of exponent functions is introduced in \cite[Chapter 2]{cru}. 
Given $p(\cdot) \in \mathcal{P}(\mathbb{R}^n)$ and a set $E \subseteq \mathbb{R}^n$, we define:
\[
p_-(E) = \operatorname*{ess~inf}_{x \in E} p(x), \quad p_+(E) = \operatorname*{ess~sup}_{x \in E} p(x).
\]
For simplicity, we write $p_- \coloneqq p_-(\mathbb{R}^n)$ and $p_+ \coloneqq p_+(\mathbb{R}^n)$. It is clear that $1 \leq p_- \leq p_+ \leq \infty$. 
The conjugate exponent function $p'(\cdot)$ of $p(\cdot) \in \mathcal{P}(\mathbb{R}^n)$ is defined by:
\[
\frac{1}{p(x)} + \frac{1}{p'(x)} = 1, \quad x \in \mathbb{R}^n.
\]
Since $p(\cdot)$ is a function, the notation $p'(\cdot)$ might be confused with the derivative of $p(\cdot)$, but we clarify that we will never use the symbol ``$^\prime$" in this sense.

Now we introduce the definition of Lebesgue spaces with variable exponents, denoted by $L^{p(\cdot)}(\mathbb{R}^n)$, as given in \cite[Definition 2.6]{cru}, which play an important role in the treatment of partial differential equations with non-standard growth conditions.

\begin{definition}[Modular]
	Given $p(\cdot)\in \mathcal{P}(\mathbb{R}^n)$ and a Lebesgue measurable function $f$, define the modular functional (or simply the modular) associated with $p(\cdot)$ by
	$$\rho_{p(\cdot)}(f) = \int_{\mathbb{R}^n\backslash \mathbb{R}_{\infty}^n} |f(x)|^{p(x)} dx +\|f\|_{L^\infty(\mathbb{R}^n_\infty)},$$
	where $\mathbb{R}_\infty^n \coloneqq  \{x\in \mathbb{R}^n: p(x)=\infty\}$. For simplicity, denote $\rho_{p(\cdot)}(f)$ as $\rho_{p}(f)$.
\end{definition}

With the modular, the Lebesgue space with variable exponent is defined as
$$L^{p(\cdot)}(\mathbb{R}^n) = \big\{f\in \mathcal{M}(\mathbb{R}^n): \|f\|_{L^{p(\cdot)}(\mathbb{R}^n)} \coloneqq  \inf\left\{\lambda>0: \rho_p({f}/{\lambda}) \leq 1\right\} <\infty \big\}.$$
Notice that if the exponent function $p(x)\equiv p$ with $1\leq p\leq \infty$, then the Lebesgue space with variable exponent coincides with the classical Lebesgue space. Naturally, some results on Lebesgue space are hoped to be extended to this space. In the past 30 years, the Lebesgue spaces with variable exponents $L^{p(\cdot)}(\mathbb{R}^n)$ have attracted the attention of numerous scholars and achieved fruitful results. For the coherence of the text, we now present some basic properties of $L^{p(\cdot)}(\mathbb{R}^n)$. The first one is completeness. 

With the modular, the Lebesgue space with variable exponents is defined as
\[
L^{p(\cdot)}(\mathbb{R}^n) = \left\{ f \in \mathcal{M}(\mathbb{R}^n) : \|f\|_{L^{p(\cdot)}(\mathbb{R}^n)} \coloneqq \inf\left\{ \lambda > 0 : \rho_p\left( \frac{f}{\lambda} \right) \leq 1 \right\} < \infty \right\}.
\]
Note that if the exponent function $p(x) \equiv p$ with $1 \leq p \leq \infty$, the Lebesgue space with variable exponents coincides with the classical Lebesgue space. Naturally, it is hoped that certain results for Lebesgue spaces can be extended to this setting. Over the past three decades, the Lebesgue spaces with variable exponents, $L^{p(\cdot)}(\mathbb{R}^n)$, have garnered significant attention, and substantial progress has been made in this area of research.

 For clarity and coherence, we now present some basic properties of $L^{p(\cdot)}(\mathbb{R}^n)$. The first property we address is completeness.
As is well known, the space $L^{p}(\mathbb{R}^n)$ is complete for $1 \leq p \leq \infty$. It is natural to ask whether $L^{p(\cdot)}(\mathbb{R}^n)$ possesses the same property. The following lemma provides a positive answer.
\begin{lemma}[\!\!\cite{cru}]\label{ref5-thm2.71}
	Given $p(\cdot)\in \mathcal{P}(\rn)$, $L^{p(\cdot)}(\mathbb{R}^n)$ is complete, that is, every Cauchy sequence in $L^{p(\cdot)}(\rn)$ converges in norm.
\end{lemma} 
The following H\"older's inequality in $L^{p(\cdot)}(\mathbb{R}^n)$ is a generalization of the classical result in Lebesgue spaces.
\begin{lemma}\label{lem5}
	Let $m\geq2, p(\cdot),p_{i}(\cdot)\in \mathcal{P}(\mathbb{R}^n)$ $(i=1,\dots,m)$, satisfying $\frac{1}{p(x)}=\sum_{i=1}^m \frac{1}{p_i(x)}$. Then for any $f_i\in L^{p_i(\cdot)}(\mathbb{R}^n)$, $\prod_{i=1}^m f_i\in L^{p(\cdot)}(\mathbb{R}^n)$. Moreover, there exists a constant $C>0$ such that 
	$$\left\|\prod_{i=1}^m f_i\right\|_{L^{p(\cdot)}(\rn)}\leq C\prod_{i=1}^m \|f_i\|_{L^{p_i(\cdot)}(\rn)}.$$
\end{lemma}
Before proving Lemma \ref{lem5}, we first present the following result from \cite[Corollary 2.28]{cru}.
\begin{lemma}[\!\!\cite{cru}]\label{ref5-cor2.28}
	Given exponent functions $r(\cdot), q(\cdot)\in \mathcal{P}(\rn)$, define $p(\cdot)\in \mathcal{P}(\rn)$ by
	$$\frac{1}{p(x)}=\frac{1}{q(x)}+\frac{1}{r(x)}.$$
	Then there exists a constant $C$ such that for all $f\in L^{q(\cdot)}(\rn)$ and $g\in L^{r(\cdot)}(\rn)$, $fg\in L^{p(\cdot)}(\rn)$ and
	$$\|fg\|_{L^{p(\cdot)}(\rn)} \le C\|f\|_{L^{q(\cdot)}(\rn)} \|g\|_{L^{r(\cdot)}(\rn)}.$$
\end{lemma}

\begin{proof}[Proof of Lemma \ref{lem5}]
The idea for the proof is based on \cite{cru}. By induction, the case for $m = 2$ is established by Lemma \ref{ref5-cor2.28}. We assume that the conclusion holds for $m = k$ and proceed with the proof for $m = k+1$.
	
	Define $r_k(x) = \frac{\prod_{i=1}^k p_{i}(x)}{\sum_{i=1}^k \prod_{j\neq i} p_j(x)}$. Notice that $\frac{1}{p(x)} = \sum_{i=1}^{k+1} \frac{1}{p_i(x)}$ with $p(\cdot) \in \mathcal{P}(\mathbb{R}^n)$, which implies that $\frac{1}{r_k(x)} + \frac{1}{p_{k+1}(x)} = \frac{1}{p(x)}$ and $r_k(\cdot) \in \mathcal{P}(\mathbb{R}^n)$, given that $r_k(x) \geq p(x) \geq 1$. Repeated application of Lemma \ref{ref5-cor2.28} results in
	\begin{align*}
		\left\|\prod_{i=1}^{k+1} f_i\right\|_{L^{p(\cdot)}(\rn)} &\leq C\left\|\prod_{i=1}^k f_i\right\|_{L^{r_k(\cdot)}(\rn)} \left\|f_{k+1}\right\|_{L^{p_{k+1}(\cdot)}(\rn)}\\
		&\leq C \prod_{i=1}^{k+1} \left\| f_i\right\|_{L^{p_{i}(\cdot)}(\rn)},
	\end{align*}
	which completes the proof.
\end{proof}
Although the basic theory of spaces with variable exponents only requires that $p(\cdot)$ is a measurable function, to ensure that operators of interest, such as the Hardy-Littlewood maximal operator, are bounded in these spaces, we often assume that $p(\cdot)$ satisfies additional regularity conditions. Two important continuity conditions are locally  $\log$-H\"older continuity and $\log$-H\"older decay condition.

\begin{definition}[$\log$-H\"older]
	Given a function $\alpha: \mathbb{R}^n\to \mathbb{R}$. The function $\alpha$ is said locally $\log$-H\"older continuous on $\mathbb{R}^n$ if there exists $C_1>0$ such that 
	$$|\alpha(x)-\alpha(y)|\leq \frac{C_1}{\log(e+1/|x-y|)}$$
	for all $x,y\in \mathbb{R}^n$. We say that $\alpha$ satisfies the $\log$-H\"older decay condition if there exists $\alpha_\infty\in \mathbb{R}$ and another constant $C_2$ such that 
	$$|\alpha(x)-\alpha_{\infty}| \leq \frac{C_2}{\log (e+|x|)}$$
	for all $x\in \mathbb{R}$. And $\alpha$ is called globally $\log$-H\"older continuous on $\mathbb{R}^n$ if it is locally $\log$-H\"older continuous and satisfies the $\log$-H\"older decay condition. The constants $C_1$ and $C_2$ are called the local $\log$-H\"older constant and the log-H\"older decay constant, respectively. The maximum $\max\{C_1,C_2\}$ is just called the $\log$-H\"older constant of $\alpha$.
\end{definition}
We now briefly review the  $\log$-H\"older condition, which was first introduced in the context of variable exponents by Zhikov \cite{zhi}, who considered the averaging of variational problems with stochastic Lagrangians. Later on, the local  $\log$-H\"older condition was shown to be sufficient for the boundedness of many operators. For instance, \cite{cru} proves the $L^{p(\cdot)}$ boundedness with $p_{-} > 1$ and the weak $L^{p(\cdot)}$ boundedness with $p_{-} = 1$ of Hardy-Littlewood maximal operators, \cite{cru5} establishes the $L^{p(\cdot)}$ boundedness with $1 < p_{-} \leq p_{+} < \infty$ of Calder\'on-Zygmund operators, and the $L^{p(\cdot)} \to L^{q(\cdot)}$ boundedness of the fractional maximal operator where $q(\cdot) \in \mathcal{P}(\mathbb{R}^n)$ is defined by
	\begin{align*}
		\frac{1}{p(x)}-\frac{1}{q(x)}=\frac{\alpha}{n}, \quad x\in \rn,
	\end{align*}
	when $p_{+}\leq \frac{n}{\alpha}$ for some $\alpha\in [0,n)$, and the $L^{p(\cdot)} \to L^{q(\cdot)}$ boundedness with $1<p_{-}\leq p_{+} <\frac{n}{\alpha}$ of Riesz potentials. 
	However, the local $\log$-H\"older condition is not necessary for the boundedness of any of the above operators, refer to \cite{cru,cru5}. 	

Below, we present the class of variable exponents, as defined in \cite[Definition 4.1.4]{die}, which is useful in the study of function spaces with variable exponents.
\begin{definition}[Variable exponents class]
	The class of variable exponents is defined by
	$$\mathcal{P}_{\log}(\mathbb{R}^n) \coloneqq  \big\{p(\cdot)\in \mathcal{P}(\mathbb{R}^n): \frac{1}{p(\cdot)} \ \text{is globally}\ \log-\text{H\"older continuous}\big\}.$$
\end{definition}
The following two lemmas characterize the relationships between the exponent functions $p(x)$ and $p'(x)$, see \cite[Propposition 2.37]{cru} and \cite[Remark 4.1.5, Theorem 4.3.8]{die}.
\begin{lemma}[\!\!\cite{die}]\label{lem6}
	If $p\in \mathcal{P}(\rn)$ with $p_+<\infty$, then $p\in \mathcal{P}_{\log}(\rn)$ if and only if $p'\in \mathcal{P}_{\log}(\rn)$, and	$p$ is a globally $\log$-H\"older continuous if and only if $p\in \mathcal{P}_{\log}(\rn)$.
\end{lemma}
\begin{lemma}[\!\!\cite{die,cru}]\label{lem7}
	Given $p(\cdot)\in \mathcal{P}(\rn)$, the associate space of $L^{p(\cdot)}(\rn)$ is equal to $L^{p'(\cdot)}(\rn)$, that is,
	$$\left(L^{p(\cdot)}(\rn)\right)' = L^{p'(\cdot)}(\rn).$$ 
	Moreover, $\|\cdot\|_{L^{p'(\cdot)}(\rn)}$ and $\|\cdot\|_{L^{p'(\cdot)}(\rn)}'$ are equivalent norms. In addition, if $p(\cdot)\in \mathcal{P}_{\log}(\rn)$ with $p_->1$, then  for all $f\in L^{p(\cdot)}(\rn)$,there exists a constant $C>0$ such that 
	$$\|Mf\|_{L^{p(\cdot)}(\rn)} \leq C(p_-)' \|f\|_{L^{p(\cdot)}(\rn)}.$$

\end{lemma}

Obviously, the Morrey spaces with variable exponents $M_{L^{p(\cdot)}(\mathbb{R}^n)}^u$ are obtained by replacing $X$ in Definition \ref{def-mb} with $L^{p(\cdot)}(\mathbb{R}^n)$, where $p(\cdot) \in \mathcal{P}(\mathbb{R}^n)$, that is,
$$
M_{L^{p(\cdot)}(\mathbb{R}^n)}^{u} = \Big\{ f \in \mathcal{M}(\mathbb{R}^n): \|f\|_{M_{L^{p(\cdot)}(\mathbb{R}^n)}^{u}} \coloneqq \sup_{x\in \mathbb{R}^n, r>0} \frac{1}{u(x,r)} \left\|f \chi_{B(x,r)}\right\|_{L^{p(\cdot)}(\rn)} < \infty\Big\},
$$
with $u(x,r): \mathbb{R}^n \times (0,+\infty) \to (0,+\infty)$ be a Lebesgue measurable function.

Considering Lemmas \ref{ref5-thm2.71}-\ref{lem7} and Theorem \ref{thm2}, we conclude the following result.
\begin{theorem}\label{thm4}
	Let $I \coloneqq  \{i_1,\dots,i_l\} = \{1,\dots,l\} \subseteq \{1,\dots,m\}$, $\frac{1}{p(x)}=\frac{1}{p_1(x)}+\cdots+\frac{1}{p_m(x)}$ with $p_i\in \mathcal{P}_{\log}(\mathbb{R}^n)$ and $1< (p_i)_- \leq (p_i)_+ <\infty$, and $u=\prod_{i=1}^m u_i$ with $u_i \in \mathbb{W}_{L^{p_i(\cdot)}(\mathbb{R}^n)}^0, 1\leq i\leq m$. Suppose that $\vec{b} \in \mathrm{BMO}^l$ and $\mathcal{T}_{\vec{b}}$ satisfies Hypothesis \ref{hyp}. If $p\in \mathcal{P}_{\log}(\mathbb{R}^n)$ meeting $1<p_-\leq p_+<\infty$ and $u\in \mathbb{W}_{L^{p_i(\cdot)}(\mathbb{R}^n)}^0$, then
	\begin{align*}
		\left\|\mathcal{T}_{\vec{b}}(\vec{f})\right\|_{M_{L^{p(\cdot)}(\mathbb{R}^n)}^{u}} 
		\leq & C \|\vec{b}\|_{\text{BMO}} 
		\|M\|_{\mathfrak{B}_{L^{p'(\cdot), u}}}^{2+l}
		\|\mathscr{M}\|_{M_{L^{p_1(\cdot)}(\mathbb{R}^n)}^{u_1} \times \cdots \times M_{L^{p_m(\cdot)}(\mathbb{R}^n)}^{u_m} \to M_{L^{p(\cdot)}(\mathbb{R}^n)}^{u}} \\
		&\times \prod_{i=1}^m \|M\|_{M_{L^{p_i(\cdot)}(\mathbb{R}^n)}^{u_i} \to M_{L^{p_i(\cdot)}(\mathbb{R}^n)}^{u_i}} 
		\|f_i\|_{M_{L^{p_i(\cdot)}(\mathbb{R}^n)}^{u_i}}.
	\end{align*}
\end{theorem}
\begin{proof}[Proof of Theorem \ref{thm4}]
	For $i=1,\dots,m$, when $1<p_-,(p_i)_- \leq p_+, (p_i)_+<\infty$, $L^{p(\cdot)}(\rn)$ and $L^{p_i(\cdot)}(\rn)$ are all Banach spaces by invoking Lemma \ref{ref5-thm2.71}. In order to verify that $\|\chi_B\|_{X} \leq C\prod_{i=1}^m \|\chi_B\|_{X_i}$, we consider the norm of $\chi_{B(x,r)}$ in $L^{p(\cdot)}(\rn)$. For any ball $B(x,r)\subseteq \rn$, define the harmonic means of $p$ and $p'$ as 
	$$\frac{1}{p_B} = \frac{1}{|B|} \int_B \frac{1}{p(x)} dx,\quad \frac{1}{p'_B} = \frac{1}{|B|} \int_B \frac{1}{p'(x)}dx.$$
	Notice that $p,p_i\in \mathcal{P}_{\log}(\rn)$ and $(p_i)_-, p_i>1$, it follows from \cite[Page 10]{cha} that
	$$\|\chi_B\|_{L^{p(\cdot)}(\rn)} \simeq |B(x,r)|^{\frac{1}{p_B}},\quad \|\chi_{B(x,r)}\|_{L^{p'(\cdot)}(\rn)} \simeq |B(x,r)|^{\frac{1}{p'_B}}.$$
Noting that $\sum_{i=1}^m \frac{1}{p_{i,B}} = \frac{1}{p_B}$, where $\frac{1}{p_{i,B}} \coloneqq \frac{1}{|B|} \int_B \frac{1}{p_{i}(x)} dx, i=1,\dots,m$, we have
	\begin{align}\label{eq9}
		\|\chi_{B(x,r)}\|_{L^{p(\cdot)}(\rn)} 
		\simeq |B(x,r)|^{\frac{1}{p_B}} 
		= |B(x,r)|^{\sum_{i=1}^m \frac{1}{p_{i,B}}} 
		= \prod_{i=1}^m |B(x,r)|^{\frac{1}{p_{i,B}}} 
		\simeq \prod_{i=1}^m \|\chi_{B(x,r)}\|_{L^{p_i(\cdot)}(\rn)}.
	\end{align}

 Following the definition of $p_-$, we obtain that $(p_-)'=(p'(\cdot))_+$, $(p_+)'=(p'(\cdot))_-$. Thus, $p_+<\infty$ implies that $(p'(\cdot))_->1$. According to Lemma \ref{lem7}, $L^{p(\cdot)}(\rn) \in \mathbb{M}'\cap \mathbb{M}, L^{p_i(\cdot)}(\rn) \in \mathbb{M}'\cap\mathbb{M}, i=1,\dots,m$, which indicates $\mathscr{M}$ is bounded on $L^{p(\cdot)}(\rn)$. In fact, Lemma \ref{lem5} guarantee that 
	\begin{align*}
		\|\mathscr{M}(\vec{f})\|_{L^{p(\cdot)}(\rn)} \leq \left\|\prod_{i=1}^m Mf_i\right\|_{L^{p(\cdot)}(\rn)} 
		\leq C\prod_{i=1}^m \|Mf_i\|_{L^{p(\cdot)}(\rn)}
		\leq C\prod_{i=1}^m \|f_i\|_{L^{p_i}(\rn)}.
	\end{align*}
Combining Lemma \ref{lem6}, Lemma \ref{lem7} with Theorem \ref{thm2} yield that
	\begin{align*}
		\left\|\mathcal{T}_{\vec{b}}(\vec{f})\right\|_{M_{L^{p(\cdot)}(\mathbb{R}^n)}^{u}} 
		\leq & C \|\vec{b}\|_{\text{BMO}} 
		\|M\|_{\mathfrak{B}_{L^{p'(\cdot), u}}}^{2+l}
		\|\mathscr{M}\|_{M_{L^{p_1(\cdot)}(\mathbb{R}^n)}^{u_1} \times \cdots \times M_{L^{p_m(\cdot)}(\mathbb{R}^n)}^{u_m} \to M_{L^{p(\cdot)}(\mathbb{R}^n)}^{u}} \\
		&\times \prod_{i=1}^m \|M\|_{M_{L^{p_i(\cdot)}(\mathbb{R}^n)}^{u_i} \to M_{L^{p_i(\cdot)}(\mathbb{R}^n)}^{u_i}} 
		\|f_i\|_{M_{L^{p_i(\cdot)}(\mathbb{R}^n)}^{u_i}},
	\end{align*}
	which completes the proof.
\end{proof}
\begin{corollary}\label{cor2}
	Let $I \coloneqq  \{i_1,\dots,i_l\} = \{1,\dots,l\} \subseteq \{1,\dots,m\}$, $\frac{1}{p(x)}=\frac{1}{p_1(x)}+\cdots+\frac{1}{p_m(x)}$ with $p_i\in \mathcal{P}_{\log}(\mathbb{R}^n)$ and $1< (p_i)_- \leq (p_i)_+ <\infty$, and for $0\leq \theta<1$,
	$$u_i(x,r) = \|\chi_{B(x,r)}\|_{L^{p(\cdot)}(\mathbb{R}^n)}^\theta, \ 1\leq i\leq m.$$
	Suppose that $\vec{b} \in \mathrm{BMO}^l$ and $\mathcal{T}_{\vec{b}}$ satisfies Hypothesis \ref{hyp}. If $p\in \mathcal{P}_{\log}(\mathbb{R}^n)$ and $1<p_-\leq p_+<\infty$, then
	$$
	\left\|\mathcal{T}_{\vec{b}}(\vec{f})\right\|_{M_{L^{p(\cdot)}(\mathbb{R}^n)}^{u}} 
	\leq C \|\vec{b}\|_{\text{BMO}} \prod_{i=1}^m \|f_i\|_{M_{L^{p_i(\cdot)}(\mathbb{R}^n)}^{u_i}},
	$$
	where $u=\prod_{i=1}^m u_i$. In particular, if $\theta=0$, then $M_{L^{p(\cdot)}(\mathbb{R}^n)}^u = L^{p(\cdot)}(\mathbb{R}^n)$ and
	$$
	\left\|\mathcal{T}_{\vec{b}}(\vec{f})\right\|_{L^{p(\cdot)}(\mathbb{R}^n)} 
	\leq C \|\vec{b}\|_{\text{BMO}} \prod_{i=1}^m \|f_i\|_{L^{p_i(\cdot)}(\mathbb{R}^n)}.
	$$
\end{corollary}
\begin{remark}
	We make three notes on Corollary \ref{cor2}.
	\begin{enumerate}
		\item If $I=\emptyset$, then compared with \cite[Lemma 3.3]{zha}, Corollary \ref{cor2} presents a broader conclusion beyond the Hardy-Littlewood maximal operator $\mathscr{M}$, since $\mathscr{M}$ satisfies Hypothesis \ref{hyp}.
		
		\item If $I\neq \emptyset$, then the boundedness for iterated commutators of multilinear operators is still brand-new in the Lebesgue space with variable exponent $L^{p(\cdot)}(\mathbb{R}^n)$.
		
		\item In terms of function space, Corollary \ref{cor2} generalizes Theorem 4.7 in \cite{zha} since $\mathbb{W}_X^1 \subseteq \mathbb{W}_X^0$.
	\end{enumerate}
\end{remark}

\begin{proof}[Proof of Corollary \ref{cor2}]
	The proof of this corollary can be divide into two cases.\\
	(1). For $0<\theta<1$, if $u_i(x_1,r_1) \leq u_i(x_2,r_2), i=1,\dots,m$, then
	$$\|\chi_{B(x_1,r_1)}\|_{L^{p_i(\cdot)}(\rn)} \leq \|\chi_{B(x_2,r_2)}\|_{L^{p_i(\cdot)}(\rn)}.$$
	Therefore, 
	\begin{align*}
		\frac{\|\chi_{B(x_1,r_1)}\|_{L^{p_i(\cdot)}(\mathbb{R}^n)}}{u_i(x_1,r_1)} 
		=\|\chi_{B(x_1,r_1)}\|_{L^{p_i(\cdot)}(\rn)}^{1-\theta} 
		\leq \|\chi_{B(x_2,r_2)}\|_{L^{p_i(\cdot)}(\rn)}^{1-\theta}
		= \frac{\|\chi_{B(x_2,r_2)}\|_{L^{p_i(\cdot)}(\mathbb{R}^n)}}{u_i(x_2,r_2)}.
	\end{align*}
	For any $x\in \rn, r>0$, by \cite[Corollary 3.4]{ho3}, we obtain
	\begin{align}\label{eq-1207}
		\sum_{j=0}^\infty \frac{\|\chi_{B(x,r)}\|_{L^{p_{i}(\cdot)}(\rn)}}{\|\chi_{B(x,2^{j+1}r)}\|_{L^{p_{i}(\cdot)}(\rn)}} u_i(x,2^{j+1}r) \leq Cu_i(x,r),
	\end{align}
	which yields that $u_i\in \mathbb{W}_{L^{p_i(\cdot)}(\rn)}^0$. Next we turn to prove $u\in \mathbb{W}_{L^{p(\cdot)}(\rn)}^0$. In view of \eqref{eq9},
	\begin{align}\label{usimchi}
		u(x,r) = \prod_{i=1}^m u_i(x,r)
		=\left(\prod_{i=1}^m \|\chi_{B(x,r)}\|_{L^{p_i(\cdot)}(\rn)}\right)^{\theta}
		\simeq \|\chi_{B(x,r)}\|_{L^{p(\cdot)}(\rn)}^{\theta}.
	\end{align}
	Together with the estimate \eqref{eq-1207} for $u_i$, \eqref{usimchi} then yields $u\in \mathbb{W}_{L^{p(\cdot)}(\rn)}^0$. Applying Theorem \ref{thm4}, we conclude that Corollary \ref{cor2} holds for $0 < \theta < 1$.\\
	(2). For $\theta = 0$, we have $u_i = u \equiv 1$, and the spaces $M_{L^{p_i(\cdot)}(\mathbb{R}^n)}^{u_i}$ and $M_{L^{p(\cdot)}(\mathbb{R}^n)}^{u}$ remain non-trivial. For any $1 < q < (p_+)'$, \cite[Proposition 6.5]{ho1} establishes that
	$$\sum_{j=0}^{\infty} \frac{\|\chi_{B(x,r)}\|_{L^{p(\cdot)}(\rn)}}{\|\chi_{B(x,2^{j+1}r)}\|_{L^{p(\cdot)}(\rn)}}
	\leq C\sum_{j=0}^{\infty} 2^{-jn\left(1-\frac{1}{q}\right)} 
	<\infty.$$
Combining (1) and (2), we complete the proof of Corollary \ref{cor2}.
\end{proof}

\section{Applications to other operators}\label{pr+5}
The purpose of this section is to demonstrate that our theorems can be applied to many operators. For example, this includes multilinear maximal singular integral operators and multilinear Bochner-Riesz square functions.
 
\subsection{Multilinear maximal singular integral operator}
\ \\
\indent Let $\omega:[0, \infty) \rightarrow[0, \infty)$ be a nondecreasing function, denote $d \vec{y}=d y_1 \cdots d y_m$, we consider the corresponding commutator of the multilinear maximal singular integral operator $T^{\ast}$, which is defined as
$$
T^{\ast}(\vec{f})(x):=\sup _{\delta>0}\left|\int_{\sum_{i=1}^m\left|y_i-x\right|^2>\delta^2} K\left(x, y_1, \ldots, y_m\right) f_1\left(y_1\right) \cdots f_m\left(y_m\right) d \vec{y}\right|,
$$
for $x \notin \bigcap_{j=1}^m \operatorname{supp} f_j$ and each $f_j \in L_c^{\infty}\, (j=1, \ldots, m)$, where the kernel function $K$, defined away from the diagonal $x=y_1=\cdots=y_m$ in $(\mathbb{R}^n)^{m+1}$, satisfies the following conditions:
\begin{enumerate}[(i)]
	\item (the size condition)
	\begin{align}\label{size}
		\left|K\left(x, \vec{y}\right)\right| \lesssim \frac{1}{\left( \sum_{k=1}^{m} \left|x-y_{k}\right| \right)^{m n}},
	\end{align}
	for all $x, y_1, \dots, y_m \in \mathbb{R}^n$ with $x \neq y_j$ for some $j \neq 0$,
	\item (the regularity condition)
	for some $0<\tau<1$, 
	\begin{align*}
		& \left|K\left(x, y_1, \ldots, y_m\right)-K\left(x^{\prime}, y_1,\ldots, y_m\right)\right| \\
		& \qquad \lesssim \frac{1}{\left(\left|x-y_1\right|+\cdots+\left|x-y_m\right|\right)^{m n}} \omega\left(\frac{\left|x-x^{\prime}\right|}{\left|x-y_1\right|+\cdots+\left|x-y_m\right|}\right)
	\end{align*}
	whenever $\left|x-x'\right| \leq  \tau \max_{1 \leq  k \leq  m} |x-y_k|$, and
	\begin{equation}
		\begin{aligned}\label{smooth-y}
			& \left|K\left(x, y_1, \ldots, y_i, \ldots, y_m\right)-K\left(x, y_1, \ldots, y_i^{\prime}, \ldots, y_m\right)\right| \\
			& \qquad\lesssim \frac{1}{\left(\left|x-y_1\right|+\cdots+\left|x-y_m\right|\right)^{m n}} \omega\left(\frac{|y_i-y_i^{\prime}|}{\left|x-y_1\right|+\cdots+\left|x-y_m\right|}\right)
		\end{aligned}	
	\end{equation}
	whenever $\left|y_j-y_j'\right| \leq  \tau \max_{1 \leq  k \leq  m} |x-y_k|$, where $\vec{y}'=(y_1,\dots,y_j',\dots,y_m)$, $j \in \{1,\dots,m\}$. 
\end{enumerate}
The commutator of the multilinear maximal singular integral operator is as follows,
$$
T_{\vec{b}}^{\ast}(\vec{f})(x):=\sum_{j=1}^m \sup _{\delta>0}\left|\int_{\sum_{i=1}^m\left|y_i-x\right|^2>\delta^2}\left(b_j(x)-b_j\left(y_j\right)\right) K\left(x,\vec{y}\right) f_1\left(y_1\right) \cdots f_m\left(y_m\right) d \vec{y}\right|.
$$
In 2023, Zhang \cite{zhang-c} showed that the commutator $T_{\vec{b}}^{\ast}$ can be dominated pointwise by a finite number of sparse operators, that is, $T_{\vec{b}}^{\ast}$ satisfies our Hypothesis \ref{hyp}, and $T_{\vec{b}}^{\ast}$ is bounded on product weighted Lebesgue spaces. Subsequently, the multilinear version of Coifman-Fefferman inequality, the quantitative weighted mixed weak type inequalities and the weighted modular inequalities were established in \cite{bp+tan}.

Based on \cite[Theorems 1.8]{zhang-c}, Hypothesis \ref{hyp} holds for multilinear maximal singular integral operator $T^{\ast}$, which combined with Theorems \ref{thm1} and \ref{thm2} gives the following results.
\begin{theorem}\label{thm-maximal}
	Let $T$ be an $m$-linear $\omega$-Calder\'{o}n-Zygmund operator with $\omega \in$ $\log$-Dini$(1,0).$ If $\vec{b} \in \mathrm{BMO}^m,$ then we have
	\begin{enumerate}[(a).]
		\item Let $X$ be a Banach function space and $u: \mathbb{R}^n \times (0,+\infty) \to (0,+\infty)$ be Lebesgue measurable. If $0<p<\infty$, $X' \in \mathbb{M}$ and $u\in \mathbb{W}_{X}^0$, then there exists a constant $C>0$ such that for any $\vec{f}\in \mathcal{M}^m$ and $(\mathscr{M}_{L(\log L)} (\vec{f}))^p \in M_{X}^{u}$, $|T_{\vec{b}}^{\ast}(\vec{f})|^p \in M_{X}^{u}$ and
		$$
		\left\| (T_{\vec{b}}^{\ast} (\vec{f}))^p \right\|_{M_{X}^{u}} 
		\leq  C\left(\|\vec{b}\|_{\text{BMO}}^{\ast}\right)^{p} \left\|M\right\|_{\mathfrak{B}_{X',u}}^{pl+\max\{2,p\}} \left\|\left(\mathscr{M}_{L(\log L)} (\vec{f}) \right)^p \right\|_{M_{X}^u},
		$$
		where $\|\vec{b}\|_{\text{BMO}}^{\ast} \coloneqq \max_{i\in \{1,\dots,m\}} \|b_{i}\|_{\text{BMO}}$.
		
		\item Let $0<p,p_{1},\dots,p_{m},q,q_{1},\dots,q_{m} \leq \infty$, $(L^{p,q}(\rn))' \in \mathbb{M}$, $[(L^{p_{i},q_{i}}(\rn))^{r}]'\in \mathbb{M}$, $\mathscr{M}(\vec{f})$ be bounded from $(L^{p_{1},q_{1}}(\rn))^{r} \times \cdots \times (L^{p_{m},q_{m}}(\rn))^{r}$ to $(L^{p,q}(\rn))^{r}$ where $1\leq r<\infty$. Suppose that $\|\chi_B\|_{L^{p,q}(\rn)} \leq C\prod_{i=1}^m \|\chi_B\|_{L^{p_{i},q_{i}}(\rn)}$ for any ball $B$. Then for $u=\prod_{i=1}^m u_i$ with $u_i^{\frac{1}{r}} \in \mathbb{W}_{(L^{p_{i},q_{i}}(\rn))^{r}}^0, i=1,\dots,m$, and $u^{\frac{1}{r}} \in \mathbb{W}_{(L^{p,q}(\rn))^{r}}^0$,
		\begin{itemize}
			\item when $r=1$ and $L^{p_{i},q_{i}}(\rn) \in \mathbb{M}, i=1,\dots,m$, 
			\begin{align*}
				\left\|T_{\vec{b}}^{\ast}(\vec{f})\right\|_{M_{p,q}^u} 
				&\leq C\|\vec{b}\|_{\text{BMO}}^{\ast} \|M\|_{\mathfrak{B}_{(L^{p,q}(\rn))',u}}^{2+l} \|M\|_{M_{p_{1},q_{1}}^{u_1} \times \cdots \times M_{p_{m},q_{m}}^{u_m} \to M_{p,q}^u} \\
			    &\quad \times \prod_{i=1}^m \|M\|_{M_{p_{i},q_{i}}^{u_i} \to M_{p_{i},q_{i}}^{u_i}} \|f_i\|_{M_{p_{i},q_{i}}^{u_i}};
			\end{align*}
			
			\item when $r>1$ and $(L^{p_{i},q_{i}}(\rn))' \in \mathbb{M}, i=1,\dots,m$,
			$$
			\left\| (T_{\vec{b}}^{\ast} (\vec{f}))^{r} \right\|_{M_{p,q}^{u}} \leq  C\left(\|\vec{b}\|_{\text{BMO}}^{\ast}\right)^{p} \left\|M\right\|_{M_{p_{1},q_{1}}^{u_1} \times \cdots \times M_{p_{m},q_{m}}^{u_m} \to M_{p,q}^u}^p \prod_{i=1}^m \left\| f_i^p \right\|_{M_{p_{i},q_{i}}^{u_i}}.
			$$
		\end{itemize}
	\end{enumerate}
\end{theorem}

\subsection{Multilinear Bochner-Riesz square function}
\ \\
\indent Given $\alpha\geq 0$ and $f_{1},\dots,f_{m} \in \mathscr{S}(\rn)$, the multilinear Bochner-Riesz square function of order $\alpha$ is defined by
\begin{align}\label{square}
	\mathcal{G}^{\alpha}(f_{1},\dots,f_{m})(x)
	\coloneqq \left(\int_{0}^{\infty} \left|\frac{\partial}{\partial R} \mathcal{B}_{R}^{\alpha+1}(f_{1},\dots,f_{m})(x)\right|^{2} R dR\right)^{\frac{1}{2}},
\end{align}
where $\frac{\partial}{\partial R} \mathcal{B}_{R}^{\alpha+1}(f_{1},\dots,f_{m})(x)$ makes sense for $\alpha>-1$ for each $R>0$, and $\mathcal{B}_{R}^{\alpha}(f_{1},\dots,f_{m})$ is the corresponding multilinear Bochner-Riesz mean
$$ \mathcal{B}_{R}^{\alpha}(f_{1},\dots,f_{m})(x)
\coloneqq \int_{(\rn)^{m}} \left(1-\frac{\sum_{i=1}^{m}|\xi_{i}|^{2}}{R^{2}}\right)_{+}^{\alpha} 
\left(\prod_{i=1}^{m} \hat{f}_{i}(\xi_{i})\right) e^{2\pi i x\cdot (\xi_{1}+\cdots+\xi_{m})} d\xi_{1}\cdots d\xi_{m}, \quad R>0.$$

It was known that \cite{cho}, 
	the square function $\mathcal{G}^{\alpha}$ can be rewritten up to a constant by
\begin{align*}
	\mathcal{G}^{\alpha}(f_{1},\dots,f_{m})(x) = \left(\int_{0}^{\infty} \left|\mathcal{K}_{R}^{\alpha} \ast (f_{1},\dots,f_{m})(x)\right|^{2} \frac{dR}{R}\right)^{\frac{1}{2}},
\end{align*}
where 
$$\mathcal{K}_{R}^{\alpha} \ast (f_{1},\dots,f_{m})(x)  = \int_{(\mathbb{R}^{n})^{m}} \mathcal{K}_{R}^{\alpha}(x-y_{1},\dots,x-y_{m})\prod_{i=1}^{m} f_{i}(y_{i}) dy_{i}$$
and 
\begin{align}\label{kra}
	\widehat{\mathcal{K}_{R}^{\alpha}}(\xi_{1},\dots,\xi_{m}) = 2(\alpha+1) \frac{\sum_{i=1}^{m}|\xi_{i}|^{2}}{R^{2}} \left(1-\frac{\sum_{i=1}^{m}|\xi_{i}|^{2}}{R^{2}}\right)_{+}^{\alpha}.
\end{align}

Note that when $m=1$, $\mathcal{G}^{\alpha}$ coincides with the Stein's square function, which was introduced in \cite{ste} and is of great significance in the study of maximal Fourier multiplier operators. 
When $m=2$, $\mathcal{G}^{\alpha}$ is the bilinear Bochner-Riesz square function. The necessary conditions and sufficient conditions on the exponents $p_{1}, p_{2}, p$ and $\alpha$ for the following boundedness were achieved in \cite{cho},
$$\left\|\mathcal{G}^{\alpha}(f_{1},f_{2})\right\|_{L^{p}(\rn)} \lesssim \|f_{1}\|_{L^{p}(\rn)} \|f_{2}\|_{L^{p}(\rn)}.$$

The boundedness of the bilinear Bochner-Riesz mean  $\mathcal{B}_{R}^{\alpha}(f_{1},f_{2})$ with $R=1$ (simply denoted by $\mathcal{B}^{\alpha}$) has been studied by many scholars. 
For dimension $n=1$, from \cite{gra1}, if $1<p\leq2$ and $p_{1},p_{2}\geq 2$, then 
\begin{align}\label{b-r m}
	\left\|\mathcal{B}^{\alpha}(f_{1},f_{2})\right\|_{L^{p}(\mathbb{R})} \lesssim \|f_{1}\|_{L^{p_{1}}(\mathbb{R})} \|g\|_{L^{p_{2}}(\mathbb{R})}
\end{align} 
for $\alpha=0$. However, for dimension $n\geq 2$, Diestel and Grafakos \cite{dies} showed that if exactly one of $p_{1},p_{2}$ or $p'$ is less than $2$, then \eqref{b-r m} become invalid for $\alpha=0$. 
Later on, for dimension $n=1$, \eqref{b-r m} was established by Bernicot et al. \cite{ber} when $\alpha>0$ and $1\leq p_{1},p_{2},p\leq \infty$, which was the first time to consider the case $\alpha>0$. 

For the bilinear Bochner-Riesz square function, using \cite[Theorem 7.2]{cho} along with Theorems \ref{thm1} and \ref{thm2}, we obtain the following result.
\begin{theorem}\label{thm-square}
	Let $\mathcal{G}^{\alpha}$ be the bilinear Bochner-Riesz square function defined by \eqref{square}, $\alpha>n-\frac{1}{2}$ and $\mathcal{K}$ be given by \eqref{kra}. Then the following statements are true:
	\begin{enumerate}[(a).]
		\item Let $X$ be a Banach function space and $u: \mathbb{R}^n \times (0,+\infty) \to (0,+\infty)$ be Lebesgue measurable. If $0<p<\infty$, $X' \in \mathbb{M}$ and $u\in \mathbb{W}_{X}^0$, then there exists a constant $C>0$ such that for any $\vec{f}\in \mathcal{M}^m$ and $(\mathscr{M} (\vec{f}))^p \in M_{X}^{u}$, $|\mathcal{G}^{\alpha}(\vec{f})|^p \in M_{X}^{u}$ and
		$$
		\left\| (\mathcal{G}^{\alpha} (\vec{f}))^p \right\|_{M_{X}^{u}} \leq  C \left\|M\right\|_{\mathfrak{B}_{X',u}}^{\max\{2,p\}} \left\|\left(\mathscr{M} (\vec{f}) \right)^p \right\|_{M_{X}^u}.
		$$
		
		\item Let $X, X_1$ and $X_2$ be Banach function spaces with $\|\chi_B\|_{X} \leq C\|\chi_B\|_{X_1} \|\chi_B\|_{X_2}$ for any ball $B$, and $X' \in \mathbb{M}$, $(X_i^p)'\in \mathbb{M}$, $\mathscr{M}(\vec{f})$ be bounded from $X_{1}^{p} \times X_{2}^{p}$ to $X^p$ where $1\leq p<\infty$. Then for $u= u_{1} u_{2}$ with $u_{1}^{\frac{1}{p}} \in \mathbb{W}_{X_{1}^p}^0, u_{2}^{\frac{1}{p}} \in \mathbb{W}_{X_{2}^p}^0$, and $u^{\frac{1}{p}} \in \mathbb{W}_{X^p}^0$,
		\begin{itemize}
			\item when $p=1$ and $X_{1}, X_{2} \in \mathbb{M}$, 
			$$
			\left\|\mathcal{G}^{\alpha}(\vec{f})\right\|_{M_{X}^u} \leq C \|M\|_{\mathfrak{B}_{X',u}}^{2} \|\mathscr{M}\|_{M_{X_1}^{u_1} \times M_{X_2}^{u_2} \to M_{X}^u} 
			\prod_{i=1}^2(
			\|M\|_{M_{X_i}^{u_i} \to M_{X_i}^{u_i}}
			\|f_i\|_{M_{X_i}^{u_i}});
			$$
			\item when $p>1$ and $X_{1}',X_{2}' \in \mathbb{M}$,
			$$
			\left\| (\mathcal{G}^{\alpha} (\vec{f}))^p \right\|_{M_{X}^{u}} \leq  C \left\|\mathscr{M}\right\|_{M_{X_1}^{u_1} \times M_{X_2}^{u_2} \to M_{X}^u}^p 
			\left\| f_{1}^p \right\|_{M_{X_1}^{u_1}}
			\left\| f_{2}^p \right\|_{M_{X_2}^{u_2}}.
			$$
		\end{itemize}
	\end{enumerate}
\end{theorem}

	
\end{document}